\documentclass[12pt,oneside,leqno]{amsart}

\usepackage{amsmath}
 
\usepackage{amssymb,mathrsfs,mathtools}

\usepackage[pagebackref,hypertexnames=false]{hyperref}

\newcommand{\R}{\mathbb{R}}
\newcommand{\Z}{\mathbb{Z}}
\newcommand{\Zf}[1]{\Z/{#1}\Z}
\newcommand{\Q}{\mathbb{Q}}
\newcommand{\N}{\mathbb{N}}

\newcommand{\K}{\mathbb{K}}
\newcommand{\E}{\mathbb{E}}

\newcommand{\ov}{\overline}

\newcommand{\GRH}{\mathsf{GRH}(q)}
\newcommand{\Sp}{\mathsf{Sp}(q,k)}

\newcommand{\leqs}{\leqslant}
\newcommand{\geqs}{\geqslant}

\newcommand{\floor}[1]{\left\lfloor{#1}\right\rfloor}

\newcommand{\Sm}{\mathscr{S}_q}
\newcommand{\D}{\mathfrak{D}}

\newcommand{\qchi}{q^*(\chi)}
\newcommand{\clq}{c_{\vg{\ell}}(q)}

\newcommand{\bigO}{O}

\DeclareMathOperator{\Ci}{Ci}

\newcommand{\PP}{\mathcal{P}_X^{\vg{\ell}}}
\newcommand{\PS}{\mathcal{P}_X^{*\vg{\ell}}}
\newcommand{\QQ}{\mathcal{Q}_X^{\vg{\ell}}}
\newcommand{\QS}{Q^*_X}
\newcommand{\ZZ}{\mathcal{Z}_X^{\vg{\ell}}}
\newcommand{\LL}{\mathcal{L}^{\vg{\ell}}}
\newcommand{\LLa}{\mathcal{L}^{\vg{\ell}^{(1)}}}
\newcommand{\LLb}{\mathcal{L}^{\vg{\ell}^{(2)}}}

\newcommand{\bet}{\beta_{\vg{\ell}}}
\newcommand{\betn}{\beta_{-1}}
\newcommand{\alphn}{\alpha_{-1}}
\newcommand{\dl}{d_{\vg{\ell}}}
\newcommand{\vg}{\boldsymbol}

\newtheorem{theorem}{Theorem}[section]
\newtheorem{lemma}[theorem]{Lemma}

\newtheorem{conjecture}[theorem]{Conjecture}
\newtheorem{proposition}[theorem]{Proposition}

\begin{document}
 
\title[Moments of the Hurwitz zeta function]{Moments of the Hurwitz zeta function on the critical line}
\author{Anurag Sahay}
\address{Department of Mathematics, University of Rochester \newline Rochester, NY 14627, USA}
\email{\href{mailto:anuragsahay@rochester.edu}{anuragsahay@rochester.edu}}

\begin{abstract}

We study the moments $M_k(T;\alpha) = \int_T^{2T} |\zeta(s,\alpha)|^{2k}\,dt$ of the Hurwitz zeta function $\zeta(s,\alpha)$ on the critical line, $s = 1/2 + it$ with a rational shift $\alpha \in \Q$. We conjecture, in analogy with the Riemann zeta function, that $M_k(T;\alpha) \sim c_k(\alpha) T (\log T)^{k^2}$ . Using heuristics from analytic number theory and random matrix theory, we conjecturally compute $c_k(\alpha)$. In the process, we investigate moments of products of Dirichlet $L$-functions on the critical line. We prove some of our conjectures for the cases $k = 1,2$.

\end{abstract}

\maketitle

\section{Introduction}

Estimating the moments of the Riemann zeta function $\zeta(s)$ on the critical line, 
\begin{equation*} M_k(T) = \int_T^{2T} \left|\zeta\left(\tfrac{1}{2}+it\right)\right|^{2k} \,dt, \end{equation*} 
is a classical problem in analytic number theory (see \cite[Chapter VII]{titchmarsh}). It is widely believed that $M_k(T) \sim c_k T (\log T)^{k^2}$ for all real $k \geqs 0$, where $c_k$ is a fixed positive constant depending only on $k$. This conjecture is trivial for $k = 0$, was proved by Hardy and Littlewood 
\cite{hardy1916contributions} for $k = 1$, was proved by Ingham \cite{ingham1928mean} for $k = 2$, and is wide open in all other cases. 

Despite the history and intractability of the problem, very precise conjectures for the exact value of $c_k$ exist. On the basis of number theoretic calculations, Conrey and Ghosh \cite{conreyghosh} conjectured the value of $c_k$ for $k = 3$ and by a different, but still number theoretic, method Conrey and Gonek \cite{conreygonek} conjectured the value of $c_k$ for $k = 3,4$. Finally, using heuristics modeling $\zeta(s)$ by characteristic polynomials of random matrices from the Gaussian unitary ensemble, Keating and Snaith \cite{keatingsnaithzeta} conjectured the value of $c_k$ for all $k > 0$, agreeing with the conjectures from \cite{conreyghosh} and \cite{conreygonek}.

The analogy with random matrix theory has led to many fruitful conjectures for moments of $L$-functions; see, for example, \cite{conreymomentsrecipe} and the references therein for details.

A weaker, and hence theoretically more tractable version of the above conjecture is the estimate $M_k(T) \asymp_k T(\log 	T)^{k^2}$. By work of Ramachandra \cite{ramachandra1978some,ramachandra1980some,ramachandra1980some2}, and Heath-Brown \cite{heath1981fractional}, the lower bound $M_k(T) \gg_k T(\log T)^{k^2}$ was known conditionally on the Riemann Hypothesis (RH) for $k>0$, and by work of {Radziwi\l\l} and Soundararajan \cite{radziwill2013continuous}, it was known unconditionally for all $k \geqs 1$. Recent work of Heap and Soundararajan \cite{heap2022lower} establishes the lower bound unconditionally for all $k > 0$. 

For the upper bound, Soundararajan \cite{soundararajan2009moments} had shown on RH that $M_k(T) \ll_{k,\epsilon} T (\log T)^{k^2+\epsilon}$ for every $\epsilon > 0$ and $k > 0$. Harper \cite{harper2013sharp} removed the dependence on $\epsilon$, conditionally establishing the sharp upper bound for every $k > 0$. The upper bound was known unconditionally for $k = 1/n$, $n \in \N$ due to Heath-Brown \cite{heath1981fractional}, and for $k = 1 + 1/n$, $n \in \N$ due to Bettin, Chandee and {Radziwi\l\l} \cite{bettin2017mean}. Recently, Heap, {Radziwi\l\l} and Soundararajan \cite{heap2019sharp} subsumed both of these results by proving the upper bound unconditionally for $0 \leqs k \leqs 2$.

The object of this paper is to investigate analogous moments of the Hurwitz zeta function, $\zeta(s,\alpha)$. For $0 < \alpha \leqs 1$ and for $\Re s > 1$, $\zeta(s,\alpha)$ is defined by the series
\begin{equation*} \zeta(s,\alpha) = \sum_{n=0}^\infty \frac{1}{(n+\alpha)^s}. \end{equation*}
As with the Riemann zeta function, the Hurwitz zeta function can be continued to a meremorphic function on the entire complex plane with a simple pole at $s = 1$ satisfying the Hurwitz relation (an analogue of the functional equation, see \cite{knoppsinaifunctional}). Clearly, $\zeta(s,1) = \zeta(s)$ and $\zeta(s,1/2) = (2^s - 1) \zeta(s)$. For these values of $\alpha$, thus, $\zeta(s,\alpha)$ has an Euler product, derived from the usual Euler product for $\zeta(s)$. However, for $0 < \alpha < 1$, $\alpha \neq 1/2$, $\zeta(s,\alpha)$ does not have an Euler product, and, in fact, the behaviour of its zeroes is very different from that of $\zeta(s)$. Spira \cite{spira1976zeros} showed that like $\zeta(s)$, $\zeta(s,\alpha)$ may have trivial zeros on the negative real line, and also that $\zeta(s,\alpha)$ is zero-free in the region $\Re s \geqs 1 + \alpha$. However, it is well-known, due to Davenport and Heilbronn \cite{davenport1936zeros} for the cases of rational or transcendental $\alpha$, and due to Cassels \cite{cassels1961footnote} for the case of algebraic irrational $\alpha$ that if $\alpha \neq 1/2,1$, then $\zeta(s,\alpha)$ always has a zero in the strip $1 < \Re s < 1 + \delta$ for every $\delta > 0$. Voronin showed \cite{voronin1976zeros} that for rational $\alpha \neq 1/2, 1$, and fixed $\sigma_1,\sigma_2$ with $1/2 < \sigma_1 < \sigma_2 < 1$, there are infinitely many zeros of $\zeta(s,\alpha)$ satisfying $\sigma_1<\Re s < \sigma_2$. This result was established also for transcendental $\alpha$ by Gonek \cite{gonekthesis}. Finally, Gonek \cite{gonek1981zeros} showed that if $\alpha = a/q$ with $(a,q) = 1$, and $\varphi(q) = 2$, where $\varphi$ is Euler's totient function, then $\zeta(s,\alpha)$ has a positive proportion of its nontrivial zeros off the critical line $\Re s = 1/2$. All of these are in contrast to the expected behaviour of $\zeta(s)$.

%A related question is whether the Hurwitz zeta function exhibits any universality properties, similar to that established for the Riemann zeta function by Voronin \cite{voronin1975universality}. Results in this direction were proven independently by Gonek \cite{gonekthesis} and Bagchi \cite{bagchi1981statistical} for $\alpha$ rational or transcendental, and remain open in the algebraic irrational case.

To study the moments of the Hurwitz zeta functions in the critical line, we define in analogy with $M_k(T)$, 
\begin{equation} M_k(T;\alpha) = \int_T^{2T} \left|\zeta\left(\tfrac{1}{2}+it,\alpha\right)\right|^{2k} \,dt, \label{def:Mkalpha}\end{equation}
so that $M_k(T;1) = M_k(T)$. One might expect the following:

\begin{conjecture} 
\label{conj: hurwitzmoments}
Let $0 < \alpha \leqs 1$ be a fixed rational and $k > 0$ be a fixed real number. Then for some constant $c_k(\alpha)$, we have
\begin{equation*} M_k(T;\alpha) \sim c_k(\alpha) T (\log T)^{k^2} \end{equation*}
as $T \to \infty$.
\end{conjecture}

When $k = 1$, this is a theorem due to Rane \cite[Theorem 2]{rane1980hurwitz}, with $c_1(\alpha) = 1$. In fact, he proved for $0<\alpha \leqs 1$ (not necessarily rational),
%\begin{equation} \label{eqn: rane} \begin{split} M_1(T;\alpha) &{}= \int_T^{2T} \left|\zeta\left(\tfrac{1}{2}+it,\alpha\right)\right|^2 \,dt \\&{}= T\log T + B(\alpha) T - \frac{1}{\alpha} + \bigO\left(\frac{T^{1/2}\log T}{\alpha^{1/2}}\right) \end{split} \end{equation}
\begin{equation} \label{eqn: rane} \begin{split} M_1(T;\alpha) {}= T\log T + B(\alpha) T - \frac{1}{\alpha} + \bigO\left(\frac{T^{1/2}\log T}{\alpha^{1/2}}\right) \end{split} \end{equation}
uniformly in $\alpha$ and $T$ with an effective constant $B(\alpha)$. This was improved further by several authors, with the current best error term due to Zhan \cite[Theorem 2]{zhan1992mean}.

Winston Heap (private communication) has indicated that the above conjecture may not extend to irrational shifts $\alpha$ if $k \neq 1$, and the true behaviour in this case appears quite delicate. Heap and the author are currently exploring the higher moments and distributions of Hurwitz zeta functions with irrational shifts in an ongoing project. A Diophantine problem with a paucity of non-diagonal solutions connected to these moments was considered by Heap, the author, and Wooley in \cite{heap2022paucity} (see also \cite[Theorem 26]{bourgain2014multiplicative}).

For $k = 2$, the conjecture can be proved using methods for fourth moments of $L$-functions of degree $1$. This was done in an unpublished section of Andersson's thesis \cite[pp. 71-72]{anderssonthesis}. We restate and reprove this result here for convenience:

\begin{theorem} \label{thm: rationalfourthmoment}
Let $a,q \geqs 1$ be fixed integers with $(a,q) = 1$, $1 \leqs a \leqs q$. Then, for $\alpha = a/q$,
\begin{equation*} \begin{split} M_2(T;\alpha) = \int_T^{2T}\left|\zeta\left(\tfrac{1}{2}+it,\alpha\right)\right|^4\,dt \sim \frac{T(\log T)^4}{2\pi^2 q}\prod_{p \mid q}\left(1 - \frac{1}{p+1}\right),\end{split}  \end{equation*}
as $T \to \infty$. That is, Conjecture~\ref{conj: hurwitzmoments} is true for $k = 2$ and $\alpha = a/q$, with
\begin{equation*} c_2(\alpha) = \frac{1}{2\pi^2 q}\prod_{p \mid q}\left(1 - \frac{1}{p+1}\right)= \frac{c_2}{q}\prod_{p \mid q}\left(1 - \frac{1}{p+1}\right), \end{equation*}
where $c_2 = c_2(1) = 1/(2\pi^2)$ is the usual proportionality constant for the fourth moment of $\zeta(s)$. More precisely, we have 
\begin{equation*} M_2(T;\alpha) = c_2(\alpha) T(\log T)^4 + O_q(T(\log T)^{3}). \end{equation*}
\end{theorem}

We show later that this agrees with our conjecture for $c_k(\alpha)$. In principle, one could also work out the lower order terms in this asymptotic. 

Our goal for the rest of the paper is to provide evidence for Conjecture~\ref{conj: hurwitzmoments} when $k\in\N$. In this case, $M_k(T;\alpha)$ can be related to the mean square of products of Dirichlet $L$-functions on the critical line. 

To explain this connection, we fix some notation that will be used throughout the paper. We assume $\alpha = a/q$ with $(a,q) = 1$ and $1 \leqs a \leqs q$. Dirichlet characters will be denoted $\chi$ or $\nu$, and will be modulo $q$ unless noted otherwise. We will use bolded, lower case (Greek or Latin) letters such as $\vg{\ell}$ for tuples of natural numbers indexed by characters modulo $q$. Thus, if $\vg{\ell}$ is such a tuple, we think of it as a function $\vg{\ell}: \mathcal{D}(q) \to \N$ where $\mathcal{D}(q)$ is the set of Dirichlet characters modulo $q$. We denote $\vg{\ell}(\chi)$ as $\ell_\chi$. Further, we define,

\begin{equation*} |\vg{\ell}| = \sum_\chi \ell_\chi,\,\lambda(\vg{\ell}) = \sum_\chi \ell_\chi^2,\, \LL(s) = \prod_\chi L(s,\chi)^{\ell_\chi}. \end{equation*}
Here, and later, sums and products over $\chi$ or $\nu$ run over $\mathcal{D}(q)$. If $\vg{\ell}$ is clear from context, we suppress it and denote $\lambda(\vg{\ell})$ simply as $\lambda$. Finally, we denote by $\dl(n)$ the coefficient of $n^{-s}$ in the Dirichlet series expansion of $\LL(s)$.

%\begin{equation*} \PP(s) = \prod_\chi P_X(s,\chi)^{\ell_\chi}, \end{equation}*}
%\begin{equation*} \ZZ(s) = \prod_\chi Z_X(s,\chi)^{\ell_\chi}. \end{equation*}

To see how products of the form $\LL(s)$ arise naturally in considering Conjecture~\ref{conj: hurwitzmoments} for $\alpha = a/q$, we observe that for $\Re{s} > 1$, the orthogonality of Dirichlet characters gives
%\begin{equation*} \begin{split} \zeta(s,\alpha) &{} = \sum_{n=0}^\infty \frac{1}{(n+\alpha)^s} \\&{} = \sum_{n=0}^\infty \frac{q^s}{(qn + a)^s} \\  &{} = q^s \sum_{m=1}^\infty \frac{1}{m^s}\left(\frac{1}{\varphi(q)} \sum_\chi \ov{\chi}(a) \chi(m)\right) \\ &{} = \frac{q^s}{\varphi(q)}\sum_\chi \ov{\chi}(a) L(s,\chi). \end{split}\end{equation*}
\begin{equation*} \begin{split} \zeta(s,\alpha) &{} =  \frac{q^s}{\varphi(q)}\sum_\chi \ov{\chi}(a) L(s,\chi). \end{split}\end{equation*}
By analytic continuation, this equality holds everywhere. Thus, by the multinomial theorem
\begin{equation}  \begin{split} \label{eqn: hurwitzdecomp}
|\zeta(s,\alpha)|^{2k} &{} = \bigg|\frac{q^{ks}}{\varphi(q)^k} \sum_{|\vg{\ell}|=k} \binom{k}{\vg{\ell}} \prod_{\chi} \bigg\{\ov{\chi}(a) L(s,\chi)\bigg\}^{\ell_\chi}\bigg|^2 \\ &{} = \frac{q^{2k\sigma}}{\varphi(q)^{2k}} \sum_{\substack{|\vg{\ell}^{(1)}| = k \\|\vg{\ell}^{(2)}|=k}}\binom{k}{\vg{\ell}^{(1)}} \binom{k}{\vg{\ell}^{(2)}} s(a;\vg{\ell}^{(1)},\vg{\ell}^{(2)})\LLa(s)\ov{\LLb(s)},\end{split} \end{equation} 
%\begin{equation} \label{eqn: hurwitzdecomp} \begin{split} 
%|\zeta(s,\alpha)|^{2k} &{} = \left|\frac{q^{ks}}{\varphi(q)^k} \sum_{|\vg{\ell}|=k} \binom{k}{\vg{\ell}} \prod_{\chi} \Bigg\{\ov{\chi}(a) L(s,\chi)\Bigg\}^{\ell_\chi}\right|^2 \\ &{} = \frac{q^{2k\sigma}}{\varphi(q)^{2k}} \sum_{\substack{|\vg{\ell}^{(1)}| = k \\|\vg{\ell}^{(2)}|=k}}\binom{k}{\vg{\ell}^{(1)}} \binom{k}{\vg{\ell}^{(2)}} \left[\prod_{\chi} \chi(a)^{\ell^{(2)}_{\chi} - \ell^{(1)}_{\chi}}\right] \LLa(s)\ov{\LLb(s)},\end{split} \end{equation} 
where $\binom{k}{\vg{\ell}}=k!/\prod_\chi \ell_\chi!$ are multinomial coefficients, the sums runs over $\vg{\ell}$ such that $|\vg{\ell}| = \sum_{\chi} \ell_{\chi} = k$, and $s(a;\vg{\ell}^{(1)},\vg{\ell}^{(2)}) = \prod_{\chi} \chi(a)^{\ell^{(2)}_{\chi} - \ell^{(1)}_{\chi}}$.
In particular, when we integrate both sides from $1/2 + iT$ to $1/2 + i2T$, the terms of this sum whose phase oscillates will probably not contribute to the main term. The terms that do not have oscillations correspond to the diagonal terms $\vg{\ell}^{(1)} = \vg{\ell}^{(2)}$ where the phases of each term in the product cancel out, yielding a positive real number. Thus, heuristically, 
\begin{equation} \label{eqn: heuristicanswer}\begin{split} M_k(T;\alpha) 
	&{} = \int_T^{2T} \left|\zeta\left(\tfrac{1}{2}+it,\alpha\right)\right|^{2k}\,dt \\ 
	&{} \approx \frac{q^k}{\varphi(q)^{2k}} \sum_{|\vg{\ell}| = k} \binom{k}{\vg{\ell}}^2 \int_T^{2T} \left|\LL\left(\tfrac{1}{2}+it\right)\right|^2 \,dt. \end{split} \end{equation}
whence, the problem of estimating $M_k(T;\alpha)$ naturally reduces to studying the mean square of $\LL(s)$ along the critical line. 

To study such moments, we will use a hybrid Euler-Hadamard product, a tool introduced originally by Gonek, Hughes and Keating \cite{gonek2007hybrid} in the context of the Riemann zeta function. Specifically, we will need the following version for Dirichlet $L$-functions in the $t$-aspect:

\begin{theorem} \label{thm: eulerhadamard}
Let $s = \sigma + it$ with $\sigma \geqs 0$ and $|t| \geqs 2$, let $X \geqs 2$ be a real parameter, and let $K$ be any fixed positive integer. Further, let $f(x)$ be a non-negative $C^\infty$-function of mass one supported on $[0,1]$, and set $u(x) = Xf(X\log(x/e) + 1)/x$ so that $u$ is a non-negative $C^\infty$-function of mass one supported on $[e^{1 - 1/X},e]$. Set
\begin{equation*} U(z) = \int_0^\infty u(x) E_1(z\log x)\,dx, \end{equation*}
where $E_1(z) = \int_z^\infty e^{-w}w^{-1} \,dw$ is the exponential integral.

Let $q$ be a fixed positive integer, and $\chi$ be a Dirichlet character modulo $q$ with conductor $\qchi$. Further, suppose that $\chi$ is induced by the primitive character $\chi^*$ modulo $\qchi$. Then,
\begin{equation*} L(s,\chi) = P_X(s,\chi) Z_X(s,\chi) \Bigg(1 + \bigO\Big(\frac{\log X}{X^\sigma}\Big) + \bigO_{K,f}\Big(\frac{X^{K+2}}{(|s|\log X)^K}\Big)\Bigg), \end{equation*}
where 
%\begin{equation*} P_X(s,\chi) = \Bigg\{\prod_{\substack{p\mid q\\p\nmid \qchi}}\bigg(1 - \frac{\chi^*(p)}{p^s}\bigg)\Bigg\}\exp\bigg(\sum_{n\leqs X} \frac{\chi^*(n)\Lambda(n)}{n^s\log n}\bigg), \end{equation*}
\begin{equation*} P_X(s,\chi) = \Bigg\{\prod_{p\mid q}\bigg(1 - \frac{\chi^*(p)}{p^s}\bigg)\Bigg\}\exp\bigg(\sum_{n\leqs X} \frac{\chi^*(n)\Lambda(n)}{n^s\log n}\bigg), \end{equation*}
and
\begin{equation*} Z_X(s,\chi) = \exp\Bigg(- \sum_{\substack{\rho \\0 \leqs \Re\rho \leqs 1 \\ L(\rho,\chi^*) = 0}} U((s_0 - \rho)\log X)\Bigg). \end{equation*}

The implied constants are uniform in all parameters including $q$, unless indicated otherwise. 

\end{theorem}

Such a hybrid Euler-Hadamard product was proved by Bui and Keating \cite{bui2007mean} in their study of moments in the $q$-aspect of Dirichlet $L$-functions at the central point $s = 1/2$ (see \cite[Remark 1]{bui2007mean}). Similar hybrid Euler-Hadamard products have been used in the literature for studying moments in many other contexts such as for  for orthogonal and symplectic families of $L$-functions \cite{bui2008families}; for $\zeta'(s)$ \cite{bui2015hybrid}; for the Dedekind zeta function $\zeta_{\K}(s)$ of a Galois extension $\K$ of $\Q$ \cite{heap2013hybriddedekind}; for quadratic Dirichlet $L$-functions over function fields \cite{bui2018hybridquadratic}, \cite{andrade2018truncated}; for normalized symmetric square $L$-functions associated with $SL_2(\Z)$ eigenforms \cite{djankovic2013symmetric}; and for quadratic 
Dirichlet $L$-functions over function fields associated to irreducible polynomials \cite{andrade2019hybrid}.

With $P(s,\chi)$ and $Z(s,\chi)$ as in Theorem~\ref{thm: eulerhadamard}, we define 
\begin{equation*} \PP(s) = \prod_\chi P_X(s,\chi)^{\ell_\chi},\,\ZZ(s) = \prod_\chi Z_X(s,\chi)^{\ell_\chi}. \end{equation*}
We can view $\LL(s)$ as an $L$-function of degree $|\vg{\ell}|$, $\PP(s)$ as an approximation to its Euler product, and $\ZZ(s)$ as an approximation to its Hadamard product. Roughly, Theorem~\ref{thm: eulerhadamard} implies that $\LL(s) \approx \PP(s) \ZZ(s)$.

As is usually the case with hybrid Euler-Hadamard products, $X$ mediates between the primes and zeroes; if we want to take fewer primes in the Euler product we must take more zeroes in the Hadamard product and vice-versa.

For $X$ growing relatively slowly with $T$, we expect the two terms in the decomposition $\LL(s) \approx \PP(s) \ZZ(s)$ to behave like independent random variables due to a separation of scales. This is analogous to the splitting conjecture of Gonek, Hughes and Keating \cite[Conjecture~2]{gonek2007hybrid}. Concretely, we have:

\begin{conjecture}[Splitting] \label{conj: splitting}
Let $X, T \to \infty$ with $X \ll_\epsilon (\log T)^{2-\epsilon}$. Then, for any tuple of nonnegative integers $\vg{\ell}$ indexed by characters modulo $q$, we have for $s = 1/2 + it$,
\begin{equation*} \frac{1}{T}\int_T^{2T}\left|\LL(s)\right|^2\,dt \sim \left(\frac{1}{T}\int_T^{2T}\left|\PP(s)\right|^2\,dt\right) \times \left(\frac{1}{T}\int_T^{2T} \left|\ZZ(s)\right|^2\,dt \right). \end{equation*}
\end{conjecture}

On \cite[p. 511]{gonek2007hybrid}, it is suggested that this splitting conjecture holds for a much wider range of $X$ and $T$ with $X = o(T)$. Recently, Heap \cite{heap2021splitting} has justified this suggestion. He proved on RH that the splitting conjecture for $\zeta(s)$ holds for every $k>0$ and a much wider range of $X$ provided one requires only an order of magnitude result, instead of an asymptotic. He also established the splitting conjecture for $k = 1$ and $k = 2$ for wider ranges of $X$ both with and without RH. 

The mean square of $\PP(s)$ can be computed exactly. 
\begin{theorem}

\label{thm: eulerproductmoment}

Let $k \geqs 0$ be a fixed integer and $\epsilon > 0$ be fixed. Let $\vg{\ell}$ be a tuple of nonnegative integers indexed by characters modulo $q$ such that $|\vg{\ell}| = \sum_\chi \ell_\chi = k$. Finally, suppose that $q^2 < X \ll_\epsilon (\log T)^{2 - \epsilon)}$. Then for $s = 1/2+it$,
\begin{equation*}\frac{1}{T}\int_T^{2T} |\PP(s)|^2\,dt = b(\vg{\ell}) F_X(\vg{\ell}) \left(1+ \bigO_{q,k,\epsilon}\left(\frac{1}{\log X}\right)\right) \end{equation*}
where $b(\vg{\ell})$ and $F_X(\vg{\ell})$ are given by 
\begin{equation} b(\vg{\ell}) = \prod_{p\nmid q} \left\{\left(1 - \frac{1}{p}\right)^{|\dl(p)|^2}\sum_{m=0}^\infty \frac{|\dl(p^m)|^2}{p^{m}}\right\}, \label{def:b}\end{equation} 
\begin{equation} F_X(\vg{\ell}) = (e^\gamma\log X)^{\lambda} \prod_p \left(1 - \frac{1}{p}\right)^{\lambda - |\dl(p)|^2}, \label{def:F}\end{equation}
where $\gamma$ is the Euler-Mascheroni constant, $\dl(n)$ is the coefficient of $n^{-s}$ in the Dirichlet series for $\LL(s)$, and $\lambda = \sum_\chi \ell_\chi^2$. 
\end{theorem}

One could prove a similar result uniformly in $c$ on any vertical line $\Re{s}=\sigma$ with $1>\sigma\geqs c \geqs 1/2$ given $X \ll_\epsilon (\log T)^{1/(1-c+\epsilon)}$ , but we choose not to do so for conciseness. Note that the product over $p$ in (\ref{def:F}) is conditionally convergent but not absolutely convergent. 

For the mean square of $\ZZ(s)$, we use random matrix theory to model each $L$-function appearing in the product by random unitary matrices. One expects that the matrices representing distinct $L$-functions behave independently as in \cite[Conjecture~2]{heap2013hybriddedekind}. This leads to:

\begin{conjecture} \label{conj: randommatrices}
Suppose that $X, T \to \infty$ with $X \ll_\epsilon (\log T)^{2-\epsilon}$. Then, for any tuple $\vg{\ell}$ of nonnegative integers indexed by characters modulo $q$, we have for $s = 1/2 + it$,
\begin{equation*} \frac{1}{T}\int_T^{2T} |\ZZ(s)|^{2}\,dt \sim \prod_\chi \Bigg[\frac{G(\ell_\chi+1)^2}{G(2\ell_\chi +1)} \left(\frac{\log \qchi T}{e^\gamma \log X}\right)^{\ell_\chi^2}\Bigg], \end{equation*}
where $G(\cdot)$ is the Barnes G-function, and $\qchi$ is the conductor of $\chi$.
\end{conjecture}

It is clear that one can use Conjectures~\ref{conj: splitting}~and~\ref{conj: randommatrices} together with Theorem~\ref{thm: eulerproductmoment} to get a conjectural asymptotic for $\int_T^{2T}|\LL(1/2+it)|^2\,dt$. Precisely, we get,

\begin{theorem} \label{thm: L-asymptotic}
If Conjecture~\ref{conj: splitting} and Conjecture~\ref{conj: randommatrices} are true for a tuple of nonnegative integers $\vg{\ell}$ indexed by characters modulo $q$ satisfying $|\vg{\ell}| = k$, then we have for $s = 1/2 + it$,
\begin{equation*} \frac{1}{T}\int_T^{2T} |\LL(s)|^2\, dt = (\clq+o_{q,k}(1)) \bigg\{\prod_\chi\left(\log \qchi T\right)^{\ell_\chi^2}\bigg\}, \end{equation*}
where $\clq$ is given by
\begin{equation*} \prod_{p} \Bigg\{\left(1-\frac{1}{p}\right)^{\lambda}\sum_{m=0}^\infty \frac{|\dl(p^m)|^2}{p^m}\Bigg\} \prod_\chi \frac{G(\ell_\chi + 1)^2}{G(2\ell_\chi+1)}. \end{equation*}
Here $\lambda$, and $G(\cdot)$ and $\qchi$ are the same as above.

\end{theorem}

Note that for a fixed $q$, the above says that the mean square of a product of Dirichlet $L$-functions grows as $\asymp_{k,q} T(\log T)^\lambda$. This is known for $|\vg{\ell}| \leqs 2$, and we shall show that in these cases our predicted constant matches up.

Due to the conditional hypotheses, the above theorem is really a conjecture. We note here that Heap made a similar conjecture about moments of products of $L$-functions from the Selberg class (see \cite[Section 6]{heap2013hybriddedekind}) using the recipe of Conrey, Farmer, Keating, Rubinstein and Snaith \cite{conreymomentsrecipe}. Specializing to Dirichlet $L$-functions, one can recover the above conjecture. 

He also discussed how such conjectures could be reproduced by using hybrid Euler-Hadamard products under appropriate hypotheses. However, since he has not worked out the details of this approach in this specific context, we do so here for completeness.

It is evident that our previous discussion about (\ref{eqn: heuristicanswer}) and Theorem~\ref{thm: L-asymptotic} can together be used to compute the correct value of $c_k(\alpha)$ in Conjecture~\ref{conj: hurwitzmoments}. 

\begin{theorem} \label{thm: constant}

Let $k \geqs 0$ and $a,q \geqs 1$ be fixed integers with $(a,q) = 1$, $1\leqs a\leqs q$. If Conjecture~\ref{conj: splitting} and Conjecture~\ref{conj: randommatrices} are true for all tuples of nonnegative integers $\vg{\ell}$ indexed by characters modulo $q$ satisfying $|\vg{\ell}| = k$, then Conjecture~\ref{conj: hurwitzmoments} follows for that value of $k$ and $\alpha = a/q$ with 
\begin{equation} \label{eqn: constant} c_k(\alpha) = c_k \frac{q^k}{\varphi(q)^{2k-1}} \prod_{p\mid q} \bigg\{\sum_{m=0}^\infty \binom{m+k-1}{k-1}^2 p^{-m}\bigg\}^{-1},\end{equation}
where $c_k = c_k(1)$ is the usual proportionality constant for moments of $\zeta(s)$. In other words, under the above hypotheses,
\begin{equation*}\begin{split} \int_T^{2T} \left|\zeta\left(\tfrac{1}{2}+it,\alpha\right)\right|^{2k}\,dt \sim c_k(\alpha) T(\log T)^{k^2},  \end{split}\end{equation*}
as $T \to \infty$ where $c_k(\alpha)$ is as in (\ref{eqn: constant}).

\end{theorem}

Since the current levels of technology can handle second moments and fourth moments of $\zeta(s)$ really well, it is natural to hope that we can prove Conjectures~\ref{conj: splitting}~and~\ref{conj: randommatrices} for $|\vg{\ell}|\leqs 2$. We define the Kronecker delta $\vg{\delta}^\chi$ by

\begin{equation*} \delta_{\nu}^\chi = \begin{cases} 1 & \text{ if } \chi = \nu \\ 0 & \text{ if } \chi \neq \nu. \end{cases} \end{equation*}
Then, we can prove:

\begin{theorem} \label{thm: smallk}

Conjecture~\ref{conj: splitting} and Conjecture~\ref{conj: randommatrices} hold unconditionally for $|\vg{\ell}| = 1$. In particular $|\vg{\ell}| = 1$ if and only if $\vg{\ell} = \vg{\delta}^\chi$ for some character $\chi$, in which case we have that for $s = 1/2 + it$, and $X,T \to \infty$ with $X \ll_\epsilon (\log T)^{2-\epsilon}$,
\begin{equation*}\begin{split} \frac{1}{T}\int_T^{2T}\left|L(s,\chi)\right|^2\,dt \sim \left(\frac{1}{T}\int_T^{2T}\left|P_X(s,\chi)\right|^2\,dt\right)\times \left(\frac{1}{T}\int_T^{2T} \left|Z_X(s,\chi)\right|^2\,dt \right), \end{split}\end{equation*}
and
\begin{equation} \frac{1}{T}\int_T^{2T} \left|Z_X\left(s,\chi\right)\right|^2\,dt \sim \frac{\log \qchi T }{e^\gamma \log X}. \end{equation}

\end{theorem}

The above theorem can almost certainly be extended to the case $|\vg{\ell}| = 2$. This corresponds to $\vg{\ell} = \vg{\delta}^\chi + \vg{\delta}^\nu$, and $\LL(s) = L(s,\chi)L(s,\nu)$ with $\chi$ and $\nu$ not necessarily distinct characters modulo $q$. 

We note first that some of these have already been proved. The case $\vg{\ell} = 2\vg{\delta}^{\chi_0}$ where $\chi_0$ is the principal character modulo $q$ was essentially proved by Gonek, Hughes, and Keating \cite[Theorem 3]{gonek2007hybrid}. More generally, the case $\vg{\ell} = \vg{\delta}^{\chi_0} + \vg{\delta}^\chi$ where $\chi$ is a (not necessarily primitive) quadratic Dirichlet character modulo $q$ was essentially proved by Heap \cite[Theorem~3]{heap2013hybriddedekind}. To see this, note from (\ref{eqn: zimprimitive}) that $Z_X(s,\chi)$ depends only on the primitive character $\chi^*$ modulo $\qchi$ that induces $\chi$. In particular, one can replace $L(s,\chi_0)^2$ with $\zeta(s)^2$ and $L(s,\chi_0)L(s,\chi)$ with $\zeta(s)L(s,\chi^*) = \zeta_{\K}(s)$ where $\K$ is a quadratic extension of $\Q$ and $\zeta_{\K}(s)$ is its Dedekind zeta function. Analogues of splitting for these products is precisely what was proven in these papers.

By following both these arguments, one should be able to extend to the general case $\vg{\ell} = \vg{\delta}^\chi + \vg{\delta}^\nu$. To do so, one would need a moment result for the product of two primitive Dirichlet $L$-functions and a short Dirichlet polynomial, generalizing that of \cite{heap2014twisted}. That is, we would need an asymptotic for
\begin{equation} \int_T^{2T}\left|L(s,\chi) L(s,\nu) \sum_{n\leqs T^{\theta}} \frac{a_n}{n^s}\right|^2 \,dt, \label{eqn: twistedfourthmomentgeneral}\end{equation}
where $\chi$ and $\nu$ are any primitive characters with conductor dividing $q$, and some $0<\theta<1$ sufficiently large. Such asymptotics exist in the special cases of $\zeta(s)^2$ \cite{hughes2010twistedfourth,bettin4thmoment} and $\zeta(s)L(s,\chi)$ \cite{heap2014twisted}, for any character $\chi$. Proving (\ref{eqn: twistedfourthmomentgeneral}) and the splitting conjecture for $\vg{\ell} = \vg{\delta}^\chi + \vg{\delta}^\nu$ for more general $\chi,\nu$ by using the methods of \cite{heap2013hybriddedekind}, \cite{heap2014twisted} and  \cite{gonek2007hybrid} as outlined above should be possible but long and technical. Thus, we do not pursue this here.

Note that Theorem~\ref{thm: constant} and Theorem~\ref{thm: smallk} together establish Conjecture~\ref{conj: hurwitzmoments} with $k = 1$ and $\alpha$ rational, giving an alternate proof of the leading term of Rane's asymptotic (\ref{eqn: rane}) in this case.

Lastly, as a final piece of evidence for Conjecture~\ref{conj: hurwitzmoments}, we prove the following results about upper and lower bounds:

\begin{theorem} \label{thm: upperandlower}

Let $k \geqs 0$ and $a,q \geqs 1$ be fixed integers with $(a,q) = 1$, $1\leqs a\leqs q$. If the Generalized Riemann Hypothesis (GRH) holds for every Dirichlet $L$-function modulo $q$, then for $\alpha = a/q$, $s = 1/2 + it$ and $\epsilon>0$,

\begin{equation*} T(\log T)^{k^2} \ll_{q,k} \int_T^{2T}\left|\zeta\left(s,\alpha\right)\right|^{2k} \,dt \ll_{q,k,\epsilon} T(\log T)^{k^2 + \epsilon} \end{equation*}

\end{theorem}

In principle, it should be possible to remove the $\epsilon$ in the upper bound by using the methods of Harper \cite{harper2013sharp}.

The rest of the paper is structured as follows. In Section~\ref{sec: eulerhadamard}, we sketch a proof of Theorem~\ref{thm: eulerhadamard}; in Section~\ref{sec: eulerproductmoment}, we prove Theorem~\ref{thm: eulerproductmoment}; in Section~\ref{sec: randommatrices}, we provide some evidence for Conjecture~\ref{conj: randommatrices}; in Section~\ref{sec: smallk}, we prove Theorem~\ref{thm: smallk}; and in Section~\ref{sec: hurwitzmoments}, we prove Theorems~\ref{thm: rationalfourthmoment}, \ref{thm: L-asymptotic}, \ref{thm: constant}, and \ref{thm: upperandlower} and verify our conjectured constants are correct in the known cases (viz. $|\vg{\ell}| \leqs 2$ or $k \leqs 2$).

\subsection*{Acknowledgements}
I would like to thank my adviser, Steven Gonek, for introducing me to this problem, for his encouragement and for his helpful comments on an earlier version of this manuscript. I would also like to thank Winston Heap for helpful discussions, and for pointing out that Theorem~\ref{thm: rationalfourthmoment} had been proved in Andersson's thesis \cite{anderssonthesis}. Finally, I would like to thank the anonymous referee for numerous corrections and comments that helped improve this manuscript considerably.

\section{Proof of Theorem~\ref{thm: eulerhadamard}} \label{sec: eulerhadamard}

The proof of Theorem~\ref{thm: eulerhadamard} is very similar to \cite[Theorem 1]{gonek2007hybrid} and \cite[Theorem 1]{bui2007mean}. The main difference lies in the fact that we are not assuming that the character is primitive.

Recall that if $\chi$ and $\chi^*$ are as in the theorem, then
%\begin{equation} L(s,\chi) = L(s,\chi^*)\prod_{\substack{p \mid q\\p\nmid \qchi}}\left(1 - \frac{\chi^*(p)}{p^s}\right). \label{eqn: limprimitive}\end{equation}
\begin{equation} L(s,\chi) = L(s,\chi^*)\prod_{p \mid q}\left(1 - \frac{\chi^*(p)}{p^s}\right). \label{eqn: limprimitive}\end{equation}
Further, by inspection we see that if $P(s,\cdot)$ and $Z(s,\cdot)$ are as in the theorem, then
%\begin{equation} P_X(s,\chi) = P_X(s,\chi^*)\prod_{\substack{p \mid q\\p\nmid \qchi}}\left(1 - \frac{\chi^*(p)}{p^s}\right), \label{eqn: pimprimitive}\end{equation}
\begin{equation} P_X(s,\chi) = P_X(s,\chi^*)\prod_{p \mid q}\left(1 - \frac{\chi^*(p)}{p^s}\right), \label{eqn: pimprimitive}\end{equation}
\begin{equation} Z_X(s,\chi) = Z_X(s,\chi^*). \label{eqn: zimprimitive}\end{equation}

Clearly, (\ref{eqn: limprimitive}),(\ref{eqn: pimprimitive}) and (\ref{eqn: zimprimitive}) show that we can assume without loss of generality that $\chi$ is a primitive character modulo $q$.

At this point, one can follow \cite[Theorem 1]{gonek2007hybrid} and \cite[Theorem 1]{bui2007mean} \emph{mutatis mutandis} to complete the proof. 
 
\section{Proof of Theorem~\ref{thm: eulerproductmoment}} \label{sec: eulerproductmoment}

We briefly discuss some notation for this section. Recall that $\dl(n)$ is the coefficient of $n^{-s}$ in the Dirichlet series of $\LL(s)$. $\dl(n)$ is essentially a divisor function `twisted' by the Dirichlet characters modulo $q$. We also use $d_k(n)$ for the true divisor function, i.e., the coefficient of $n^{-s}$ in $\zeta(s)^k$. In particular, it is immediate from writing $\dl(n)$ out as a convolution that $|\dl(n)|\leqs d_k(n)$ for every $n \in \N$. We will use the notation $\Sm(X)$ to denote the set of $X$-smooth (also known as $X$-friable) numbers which are coprime to $q$. That is,
\begin{equation*} \Sm(X) = \{ n \in \N : p \mid n \implies p \leqs X \text{ and } p \nmid q \}. \end{equation*}

We will need as a lemma, Mertens' theorem for arithmetic progressions:

\begin{lemma} \label{lem: apmertens}

Let $\kappa$ be a fixed real number, and $(c,q) = 1$. Then,
\begin{equation*} \prod_{\substack{p\leqs X\\p \equiv c \pmod q}} \left(1 - \frac{1}{p}\right)^{-\kappa} = H^q_c(\kappa)\left(1 + \bigO_{q,\kappa}\left(\frac{1}{\log X}\right)\right) \end{equation*}
where,
\begin{equation*} H^q_c(\kappa) = \left\{e^\gamma \log X\prod_{p} \left(1 - \frac{1}{p}\right)^{1 - \delta_q(p,c)\varphi(q)} \right\}^{\frac{\kappa}{\varphi(q)}}. \end{equation*}
Here $\gamma$ is the Euler-Mascheroni constant and $\delta_q(x,y)$ is the Kronecker delta in $\Zf{q}$,

\begin{equation*} \delta_q(x,y) = \begin{cases} 1 & \text{ if } x \equiv y \pmod q, \\ 0 &\text{ otherwise.} \end{cases} \end{equation*}
\end{lemma}
\begin{proof}

Clearly the result for general $\kappa \in \R$ follows from the case $\kappa = 1$ by exponentiating. The latter is precisely Merten's theorem for arithmetic progressions which was proved by Williams \cite{mertensforapwilliams}. The expression for the constant $H^q_c(1)$, however,  is due to Languasco and Zaccagnini \cite[Section~6]{mertensforapconstant} who also improved the error term to one uniform in $q$. The weaker form suffices for our purposes. 

\end{proof}

We also have the following, which is immediate from \cite[Lemma 3]{bui2007mean}:

\begin{lemma} \label{lem: shorteulerproduct}
Let $\vg{\ell}$ be a tuple of nonnegative integers indexed by characters modulo $q$ such that $|\vg{\ell}| = \sum_\chi \ell_\chi = k$, let
\begin{equation*} P^*_X(s,\chi) = \prod_{p\leqs X} \left(1 - \frac{\chi(p)}{p^s}\right)^{-1} \prod_{\sqrt{X} <p \leqs X} \left(1 + \frac{\chi(p)^2}{2p^{2s}}\right)^{-1}, \end{equation*}
and let
\begin{equation*} \PS(s) = \prod_{\chi} P^*_X(s)^{\ell_\chi}. \end{equation*}
Then, uniformly for $\sigma \geqs 1/2$ and $X > q^2$,
\begin{equation*} \PP(s) = \PS(s)\left(1 + \bigO_k\left(\frac{1}{\log X}\right)\right). \end{equation*}
\end{lemma}

\begin{proof} From \cite[Lemma 3]{bui2007mean}, we get that
\begin{equation*} P_X(s,\chi^*)^{\ell_\chi} = P^*_X(s,\chi^*)^{\ell_\chi}\left(1 + \bigO_{\ell_\chi}\left(\frac{1}{\log X}\right)\right), \end{equation*}
where $\chi^*$ is the primitive character modulo $\qchi$ which induces $\chi$. Since $X > q^2$, we see that $p \mid q$ implies that $p \leqs \sqrt{X}$. Thus, by inspection, 
%\begin{equation*} P_X^*(s,\chi) = P_X^*(s,\chi^*)\prod_{\substack{p \mid q\\p\nmid \qchi}}\left(1 - \frac{\chi^*(p)}{p^s}\right). \end{equation*}
\begin{equation*} P_X^*(s,\chi) = P_X^*(s,\chi^*)\prod_{p \mid q}\left(1 - \frac{\chi^*(p)}{p^s}\right). \end{equation*}
Putting the above two equalities together with (\ref{eqn: pimprimitive}), we get that
\begin{equation*} P_X(s,\chi)^{\ell_\chi} = P^*_X(s,\chi)^{\ell_\chi}\left(1 + \bigO_{\ell_\chi}\left(\frac{1}{\log X}\right)\right). \end{equation*}
The lemma follows by taking a product over characters $\chi$ modulo $q$.
\end{proof}

For the rest of this  section, we will fix $s = 1/2 + it$. Now, we want to estimate $\int_T^{2T} \left|\PP(s)\right|^2 dt$ assuming that $q^2 < X \ll_\epsilon (\log T)^{2-\epsilon}$. Clearly, by Lemma~\ref{lem: shorteulerproduct},
\begin{equation*} \frac{1}{T} \int_T^{2T} |\PP(s)|^2\,dt = \left(\frac{1}{T} \int_T^{2T} |\PS(s)|^2 \,dt\right) \left(1 + \bigO_k\left(\frac{1}{\log X}\right)\right), \end{equation*} 
and so it suffices to compute $\int_T^{2T} |\PS(s)|^2\,dt$. 

From the definition of $\PS(s)$ in Lemma~\ref{lem: shorteulerproduct}, it follows that if
\begin{equation} \PS(s) = \sum_{n=1}^\infty \frac{\bet(n)}{n^s}, \label{eqn: ppdirichlet}\end{equation}
then $\bet(n)$ is multiplicative and supported on $\Sm(X)$, $|\bet(n)| \leqs d_{2k}(n)$ for all $n$, and finally for $n \in \Sm(\sqrt{X})$ and $p\in \Sm(X)$, we have $\bet(n) = \dl(n)$ and $\bet(p) = \dl(p)$.

We truncate the sum in (\ref{eqn: ppdirichlet}) at $T^\theta$ where $\theta>0$ will be chosen later. Thus,
\begin{equation*} \PS(s) = \sum_{\substack{n \in \Sm(X) \\ n \leqs T^\theta}} \frac{\bet(n)}{n^s} + \bigO\Bigg(\sum_{\substack{n\in\Sm(X)\\n > T^\theta}} \frac{|\bet(n)|}{n^{1/2}}\Bigg). \end{equation*}
Applying Rankin's trick and the estimate $|\bet(n)| \leqs d_{2k}(n)$ to the error term, we see that it is 
\begin{equation*} \begin{split} \ll_\epsilon \sum_{\substack{n \in \Sm(X)\\ n > T^\theta}} \left(\frac{n}{T^\theta}\right)^{\epsilon} \frac{|\bet(n)|}{n^{1/2}} &{} \leqs T^{-\epsilon\theta} \sum_{n\in\Sm(X)} \frac{d_{2k}(n)}{n^{1/2 - \epsilon}}\\&{} = T^{-\epsilon \theta} \prod_{\substack{p\leqs X\\p\nmid q}} \left(1 - p^{\epsilon-1/2}\right)^{-2k}. \end{split} \end{equation*} 

Using $\log(1-x)^{-1} = \bigO(x)$, we see that the product on the right is

\begin{equation*} T^{-\epsilon\theta} \exp\left(\bigO\left(k\sum_{p\leqs X} p^{\epsilon -1/2}\right)\right). \end{equation*}

Applying the prime number theorem and integrating by parts, we see that since $X \ll_\epsilon (\log T)^{2-\epsilon}$, this is

\begin{equation*} \begin{split} & \ll T^{-\epsilon \theta} \exp\left(\bigO\left(\frac{k X^{1/2+\epsilon}}{(1/2+\epsilon)\log X}\right)\right) \\&{} \ll T^{-\epsilon\theta} \exp\left(\bigO_\epsilon\left(\frac{k\log T}{\log\log T}\right)\right) \ll_{k,\epsilon,\theta} T^{-\epsilon\theta/2}. \end{split} \end{equation*}
Hence, we have 
\begin{equation} \PS(s) = \sum_{\substack{n\in \Sm(X)\\n \leqs T^\theta}} \frac{\bet(n)}{n^s} + \bigO_{k,\epsilon,\theta}(T^{-\epsilon\theta/2}). \label{eqn: pishortdir}\end{equation}

Now, by the classical mean value theorem for Dirichlet polynomials, we have that
\begin{equation*} \int_T^{2T} \Bigg|\sum_{\substack{n\in\Sm(X)\\n \leqs T^\theta}} \frac{\bet(n)}{n^{1/2+it}}\Bigg|^2\,dt = (T + \bigO(T^\theta \log T)) \sum_{\substack{n\in \Sm(X)\\n\leqs T^\theta}} \frac{|\bet(n)|^2}{n}. \end{equation*}
Extending the sum on the right hand side to infinity introduces an error $O_{k,\epsilon,\theta}(T^{-\epsilon\theta/2})$, by the same argument as before. Thus, setting $\theta = 1/2$, we see that 
\begin{equation} \frac{1}{T} \int_T^{2T} \Bigg|\sum_{\substack{n\in\Sm(X)\\n \leqs T^{1/2}}} \frac{\bet(n)}{n^{1/2+it}}\Bigg|^2\,dt =\sum_{n\in\Sm(X)} \frac{|\bet(n)|^2}{n} (1 + O_{k,\epsilon}(T^{-\epsilon/4})) \label{eqn: dirichletmvt}\end{equation}
Using (\ref{eqn: pishortdir}) to replace $\PS(s)$ with a short Dirichlet polynomial together with (\ref{eqn: dirichletmvt}) and Cauchy-Schwarz, we conclude that
\begin{equation*} \frac{1}{T} \int_T^{2T} |\PS(s)|^2 \,dt = \sum_{n\in\Sm(X)} \frac{|\bet(n)|^2}{n} (1 + O_{k,\epsilon}(T^{-\epsilon/4})). \end{equation*}

Thus, it remains to estimate the sum $\sum_{n\in\Sm(X)} \frac{|\bet(n)|^2}{n}$. Since $\bet$ is multiplicative and supported on $\Sm(X)$, we see that

\begin{equation*}\sum_{n\in\Sm(X)} \frac{|\bet(n)|^2}{n} = \prod_{\substack{p \leqs X \\ p\nmid q}} \left( \sum_{m=0}^\infty \frac{|\bet(p^m)|^2}{p^{m}}\right). \end{equation*}
Heuristically, $\bet(n)$ was chosen to approximate $\dl(n)$. So, we expect that we can replace $\bet(p^m)$ with $\dl(p^m)$ on the right with a tolerable multiplicative error. Now, recall that $\bet(n) = \dl(n)$ when $n \in \Sm(\sqrt{X})$, and $\bet(p) = \dl(p)$ for $p \leqs X$. Thus, we can replace $\bet(p^m)$ by $\dl(p^m)$ if $p \leqs \sqrt{X}$ or $m =1$. Hence, it suffices to bound
%\begin{equation*} \prod_{\substack{\sqrt{X} < p \leqs X \\ p\nmid q}} \frac{1 + \frac{|\dl(p)|^2}{p}+ \sum_{m=2}^\infty \frac{|\bet(p^m)|^2}{p^{m}}}{\sum_{m=0}^\infty \frac{|\dl(p^m)|}{p^{m}}}. \end{equation*} 
\begin{equation*} \prod_{\sqrt{X} < p \leqs X} \frac{1 + \frac{|\dl(p)|^2}{p}+ \sum_{m=2}^\infty \frac{|\bet(p^m)|^2}{p^{m}}}{\sum_{m=0}^\infty \frac{|\dl(p^m)|}{p^{m}}}. \end{equation*} 
However, this is clearly
%\begin{equation*} \prod_{\substack{\sqrt{X} < p \leqs X \\ p\nmid q}} \left(1 + \bigO_k\left(\frac{1}{p^{2}}\right)\right) = 1 + \bigO_k\left(\frac{X^{-1/2}}{\log X}\right). \end{equation*}
\begin{equation*} \prod_{\sqrt{X} < p \leqs X} \left(1 + \bigO_k\left(\frac{1}{p^{2}}\right)\right) = 1 + \bigO_k\left(\frac{X^{-1/2}}{\log X}\right). \end{equation*}
Thus,
\begin{equation} \sum_{n\in\Sm(X)} \frac{|\bet(n)|^2}{n} = \left(1 + \bigO_k\left(\frac{X^{-1/2}}{\log X}\right) \right)\prod_{\substack{p \leqs X \\ p\nmid q}} \left( \sum_{m=0}^\infty \frac{|\dl(p^m)|^2}{p^{m}}\right). \label{eqn: momentofbeta}\end{equation}

Note that we can write the product on the right as
\begin{equation*} \prod_{\substack{p \leqs X \\ p\nmid q}} \left(\left(1 - \frac{1}{p}\right)^{|\dl(p)|^2} \sum_{m=0}^\infty \frac{|\dl(p^m)|^2}{p^{m}}\right) \prod_{\substack{p \leqs X \\ p\nmid q}} \left(1 - \frac{1}{p}\right)^{-|\dl(p)|^2} \end{equation*}
The constraint $p\leqs X$ can be removed from the first product here as that induces a multiplicative error given by
%\begin{equation*} \begin{split} \prod_{\substack{p > X \\ p\nmid q}} \left(\left(1 - \frac{1}{p}\right)^{|\dl(p)|^2} \sum_{m=0}^\infty \frac{|\dl(p^m)|^2}{p^{m}}\right)&{} = \prod_{\substack{p > X \\ p\nmid q}} \left(1 + \bigO_k\left(\frac{1}{p^{2}}\right)\right) \\&{} = 1 + \bigO_k\left(\frac{1}{X\log X}\right). \end{split} \end{equation*} 
\begin{equation*} \begin{split} \prod_{p > X} \left(\left(1 - \frac{1}{p}\right)^{|\dl(p)|^2} \sum_{m=0}^\infty \frac{|\dl(p^m)|^2}{p^{m}}\right)&{} = \prod_{p > X} \left(1 + \bigO_k\left(\frac{1}{p^{2}}\right)\right) \\&{} = 1 + \bigO_k\left(\frac{1}{X\log X}\right). \end{split} \end{equation*} 
On doing so, the expression now looks like 
\begin{equation} \label{eqn: bproduct} b(\vg{\ell},\sigma) \prod_{\substack{p \leqs X \\ p\nmid q}} \left(1 - \frac{1}{p^{2\sigma}}\right)^{-|\dl(p)|^2}. \end{equation} 

Now, define
\begin{equation*} r_\chi = \sum_{\substack{\nu,\nu' \\ \nu \ov{\nu'} = \chi}} \ell_\nu\ell_{\nu'} = \sum_{\nu} \ell_\nu \ell_{\nu\chi}. \end{equation*} 
In particular, note that $r_\chi = r_{\ov{\chi}}$ and $r_{\chi_0} = \sum_{\chi} \ell_\chi^2 = \lambda$. Further, define,
\begin{equation*} \kappa(c) = \sum_\chi r_\chi \chi(c). \end{equation*}

Clearly $\kappa(c)$ is real, and further the definition of $\dl(n)$ as a convolution gives us that
\begin{equation*} |\dl(p)|^2 = \sum_\chi r_\chi \chi(p) = \sum_\chi r_\chi \chi(c) = \kappa(c). \end{equation*}
if $p \equiv c \pmod q$. In particular, this means that the product in (\ref{eqn: bproduct}) can be divided along congruence classes modulo $q$, giving 
\begin{equation*} \prod_{(c,q) = 1} \prod_{\substack{p \leqs X \\ p\equiv c \pmod q}} \left(1 - \frac{1}{p^{2\sigma}}\right)^{-\kappa(c)}. \end{equation*}
where the outside product runs over a set of representatives of all residue classes coprime to $q$. Thus, applying Lemma~\ref{lem: apmertens}, this is
\begin{equation*} \left(1 + \bigO_q\left(\frac{1}{\log X}\right)\right) \prod_{(c,q) = 1} H^q_c(\kappa(c)). \end{equation*}

In fact, we have that $F_X(\vg{\ell}) = \prod_{(c,q) = 1} H^q_c(\kappa(c))$. To see this, note by orthogonality of characters,
\begin{equation*} \sum_{(c,q)=1} \kappa(c) = \sum_{(c,q) = 1}\sum_{\chi} r_\chi \chi(c) = r_{\chi_0} \varphi(q) = \lambda\varphi(q). \end{equation*}
Thus,
\begin{equation*} \begin{split} \prod_{(c,q)=1}H^q_c(\kappa(c)) &{} = \prod_{(c,q)=1} \left[e^\gamma \log X\prod_{p} \left(1 - \frac{1}{p}\right)^{1 - \delta_q(p,c)\varphi(q)} \right]^{\frac{\kappa(c)}{\varphi(q)}} \\& {} = (e^\gamma\log X)^{\lambda} \prod_{(c,q) = 1} \prod_p \left(1 - \frac{1}{p}\right)^{\frac{\kappa(c)}{\varphi(q)} - \delta_q(p,c)\kappa(c)} \\& {} = (e^\gamma\log X)^{\lambda} \prod_p \left(1 - \frac{1}{p}\right)^{\lambda - |\dl(p)|^2} = F_X(\vg{\ell}). \end{split} \end{equation*}

Collecting our estimates together proves Theorem~\ref{thm: eulerproductmoment}.

\section{Heuristics for Conjecture~\ref{conj: randommatrices}} \label{sec: randommatrices}

We closely follow the arguments in \cite[Section 4]{heap2013hybriddedekind} and \cite[Section 4]{gonek2007hybrid}. We want to heuristically estimate 
\begin{equation*} \frac{1}{T}\int_T^{2T} |\ZZ(s)|^{2}\,dt \end{equation*}
for $s = 1/2 + it$. The factor $Z_X(s,\chi)$ arises as a partial Hadamard product for $L(s,\chi^*)$, where $\chi^*$ is the unique primitive character that induces $\chi$. For a fixed $\chi$, $L(s,\chi^*)$ in the $t$-aspect forms a unitary family, and so we replace each $Z_X(s,\chi)$ with a unitary matrix chosen uniformly with respect to the Haar measure. 

The approximate mean density of the zeros of $L(s,\chi^*)$ in the region $0\leqs \sigma \leqs 1$ and $T \leqs t \leqs 2T$ is given by
\begin{equation*} \frac{1}{\pi}\D(\chi,T) = \frac{1}{\pi} \log\left(\frac{\qchi T}{2\pi}\right) \end{equation*}
where $\qchi$ is the conductor of $\chi$. The rescaled zeroes of $L(s,\chi^*)$ at height $T$ are well-modeled by the eigenangles of a uniformly sampled unitary matrix $\mathfrak{U}(N(\chi))$ of size $N(\chi) = \floor{\D(\chi,T)}$. 

We now assume the Generalized Riemann Hypothesis for all characters modulo $q$. Thus, the non-trivial zeros of $L(s,\chi^*)$ are of the form $1/2 + i\gamma(\chi)$ where $\gamma$ runs over a discrete (multi)set of real numbers depending on $\chi$. Now, consider the trignometric integral
\begin{equation*} \Ci(z) = -\int_z^\infty \frac{\cos w}{w} \,dw.\end{equation*}
If $E_1(z) = \int_z^\infty e^{-w} w^{-1}\, dw$ is the exponential integral as in Theorem~\ref{thm: eulerhadamard}, then $\Re\{E_1(ix)\} = - \Ci(|x|)$. 

Hence, using the definition of $\ZZ(s)$ and $Z_X(s,\chi)$,
\begin{multline*} \frac{1}{T}\int_T^{2T} \left|\ZZ\left(\tfrac{1}{2}+it\right)\right|^{2}\,dt = \frac{1}{T}\int_T^{2T} \prod_{\chi} \left|Z_X\left(\tfrac{1}{2}+it,\chi\right)\right|^{2\ell_\chi} \,dt \\{}= \frac{1}{T}\int_T^{2T} \prod_{\chi}\prod_{\gamma(\chi)} \exp\left(2\ell_\chi \int_1^e u(y) \Ci(|t - \gamma(\chi)|\log y\log X)\right) \,dy\,dt,  \end{multline*} 
where $u(y)$ is a non-negative function of mass $1$ supported in $[e^{1-1/X},e]$, as in Theorem~\ref{thm: eulerhadamard}, and we have used GRH. Now, following \cite[Equation~4.8]{heap2013hybriddedekind}, if we define $\phi(m,\theta)$ by, 
\begin{equation*} \phi(m,\theta) = \exp\left(2m \int_1^e u(y) \Ci(|\theta|\log y\log X)\right), \end{equation*}
then we see that the above integral is modeled by 
\begin{equation*} \E\left[\prod_\chi \prod_{n=1}^{N(\chi)} \phi(\ell_\chi,\theta_n(\chi))\right], \end{equation*}
where $\theta_n(\chi)$ is the $n$th eigenangle of $\mathfrak{U}(N(\chi))$. Here, the expectation is taken against the probability space from which the random matrices $\mathfrak{U}(N(\chi))$ are drawn. In particular, we make an independence assumption between the $\mathfrak{U}(N(\chi))$ for any finite set of distinct characters $\chi$, similar to \cite{heap2013hybriddedekind}. Thus, the expectation factorises, giving
\begin{equation*} \prod_\chi \E\left[\prod_{n=1}^{N(\chi)} \phi(\ell_\chi,\theta_n(\chi))\right]. \end{equation*}
We can now use \cite[Theorem 4]{gonek2007hybrid} (see also \cite[Equation~4.10]{heap2013hybriddedekind}), to compute the expectation inside. This gives us
\begin{equation*} \prod_\chi\left[ \frac{G(\ell_\chi+1)^2}{G(2\ell_\chi+1)}\left(\frac{N(\chi)}{e^\gamma \log X}\right)^{\ell_\chi^2}\left(1 + \bigO_{\ell_\chi}\left(\frac{1}{\log X}\right)\right)\right]. \end{equation*}
Finally, recall that $N(\chi) \approx \log(\qchi T)$, completing the heuristic.

\section{Proof of Theorem~\ref{thm: smallk}} \label{sec: smallk}

We begin this section by observing that to prove Theorem~\ref{thm: smallk} for $|\vg{\ell}| = 1$, it suffices to verify Conjecture~\ref{conj: randommatrices} for $|\vg{\ell}| = 1$. To see this note that $|\vg{\ell}| = 1$ is the same as $\vg{\ell} = \vg{\delta}^\chi$. Now, it is well-known (see, for example, Lemma~\ref{lem: twisted2ndmom}) that for a fixed $q$, 
\begin{equation*} \frac{1}{T}\int_T^{2T} \left|L\left(\tfrac{1}{2}+it,\chi\right)\right|^2\,dt \sim \frac{\varphi(q)}{q} \log T. \end{equation*} 
Further, putting $\vg{\ell} = \vg{\delta}^\chi$ in Theorem~\ref{thm: eulerproductmoment} gives 
\begin{equation*} \frac{1}{T}\int_T^{2T} \left|P_X\left(\tfrac{1}{2}+it,\chi\right)\right|^2 \,dt \sim \frac{\varphi(q)}{q}(e^\gamma \log X), \end{equation*}
provided that $q^2 < X \ll_\epsilon (\log T)^{2-\epsilon}$. 
Finally, Conjecture~\ref{conj: randommatrices} for $\vg{\ell} = \vg{\delta}^\chi$ states that for $X,T \to \infty$ with $X \ll_\epsilon (\log T)^{2-\epsilon}$, 
\begin{equation} \frac{1}{T}\int_T^{2T} \left|Z_X\left(\tfrac{1}{2}+it,\chi\right)\right|^2\,dt \sim \frac{\log \qchi T }{e^\gamma \log X}. \label{eqn: splitting} \end{equation}
Thus, we see that if we can prove (\ref{eqn: splitting}), then Theorem~\ref{thm: smallk} follows. 

Our first step towards proving (\ref{eqn: splitting}) is the following lemma which is a straightforward corollary of Lemma~\ref{lem: shorteulerproduct}:

\begin{lemma}

Let $\vg{\ell}$ be a tuple of nonnegative integers indexed by characters modulo $q$ such that $|\vg{\ell}| = \sum_\chi \ell_\chi = k$, define
\begin{equation*} Q_X(s,\chi) = \prod_{p\leqs \sqrt{X}} \left(1 - \frac{\chi(p)}{p^s}\right) \prod_{\sqrt{X} <p \leqs X} \left(1 -\frac{\chi(p)}{p^s} + \frac{\chi(p)^2}{2p^{2s}}\right), \end{equation*}
and define
\begin{equation*} \QQ(s) = \prod_{\chi} Q_X(s,\chi)^{\ell_\chi}. \end{equation*}
Then, uniformly for $\sigma \geqs 1/2$ and $X > q^2$,
\begin{equation*} \left[\PP(s)\right]^{-1} = \QQ(s)\left(1 + \bigO_k\left(\frac{1}{\log X}\right)\right) \end{equation*}

\end{lemma}

\begin{proof}

Clearly it suffices to restrict ourselves to $\vg{\ell} = \vg{\delta}^\chi$. Then, by Lemma~\ref{lem: shorteulerproduct},
\begin{equation*} \begin{split} P_X(s,\chi)Q_X(s,\chi) &{}= P_X^*(s,\chi)Q_X(s,\chi)\left(1+\bigO\left(\frac{1}{\log X}\right)\right) \\ &{} = \left(1+\bigO\left(\frac{1}{\log X}\right)\right) \prod_{\sqrt{X} < p \leqs X}\left(1 + \bigO\left(\frac{1}{p^{3\sigma}}\right)\right) \\&{} = 1 + \bigO\left(\frac{1}{\log X}\right), \end{split} \end{equation*} 
as desired.

\end{proof}

In view of the previous lemma and Theorem~\ref{thm: eulerhadamard}, to prove (\ref{eqn: splitting}) we want to show
\begin{equation*} \frac{1}{T}\int_T^{2T} \left|L\left(\tfrac{1}{2}+it,\chi\right) Q_X\left(\tfrac{1}{2}+it,\chi\right)\right|^2\,dt \sim \frac{\log \qchi T}{e^\gamma \log X}. \end{equation*}
Furthermore, we can assume without loss of generality that $\chi$ is primitive. To see this, let $\chi^*$ be the Dirichlet character modulo $\qchi$ which induces $\chi$. Then, $L(s,\chi)$ and $L(s,\chi^*)$ differ only by local factors corresponding to primes $p$ dividing $q$ but not dividing $\qchi$ and similarly for $X > q^2$, $Q_X(s,\chi)$ and $Q_X(s,\chi^*)$ also differ only by local factors corresponding to such $p$. In particular, we see that on multiplying these local factors cancel out, giving $L(s,\chi)Q_X(s,\chi) = L(s,\chi^*)Q_X(s,\chi^*)$. 

Thus, for $\chi$ primitive, we want to show that 
\begin{equation*} \frac{1}{T}\int_T^{2T} \left|L\left(\tfrac{1}{2}+it,\chi\right) Q_X\left(\tfrac{1}{2}+it,\chi\right)\right|^2\,dt \sim \frac{\log qT}{e^\gamma \log X}. \end{equation*}
To evaluate a mean square like this, we need a second moment asymptotic for a Dirichlet $L$-function twisted by a short Dirichlet polynomial. We use one proved by Wu \cite{wu2019twisted}:

\begin{lemma} \label{lem: twisted2ndmom}

Let $\chi$ be a primitive Dirichlet character modulo $q$ with $\log q = o(\log T)$, let $\theta>0$ be a parameter, and let $b(n)$ be an arithmetic function satisfying $b(n) \ll_\epsilon n^\epsilon$ for all $\epsilon>0$. Further, let
\begin{equation*} B_\theta(s,\chi) = \sum_{n\leqs T^\theta} \frac{\chi(n)b(n)}{n^s}, \end{equation*}
\begin{equation*} M_\theta(T;\chi,b) = \frac{1}{T}\int_T^{2T} \left|L\left(\tfrac{1}{2}+it,\chi\right)B_\theta\left(\tfrac{1}{2}+it,\chi\right)\right|^2\,dt, \end{equation*}
and
\begin{equation*} M'_\theta(T;\chi,b) = \frac{\varphi(q)}{q} \sum_{\substack{m,n\leqs T^\theta\\(mn,q)=1}} \frac{b(m)\ov{b(n)}}{[m,n]}\left(\log \frac{qT(m,n)^2}{2\pi mn} + C + \sum_{p\mid q} \frac{\log p}{p-1}\right),\end{equation*}
with $C = 2\gamma - 1 + 2\log 2$. Then,
\begin{equation*} M_\theta(T;\chi,b) = M_\theta'(T;\chi,b) + \bigO(T^{-\varepsilon_\theta}) \end{equation*}
where the parameter $\varepsilon_\theta$ depends on $\theta$, and $\varepsilon_\theta > 0$ when $\theta < 17/33$.
\end{lemma}
\begin{proof} This is contained in \cite[Theorem 1.1]{wu2019twisted}. \end{proof}

Now, writing $Q_X(s,\chi)$ as a Dirichlet series, we have
\begin{equation*} Q_X\left(s,\chi\right) = \sum_{n=1}^\infty \frac{\betn(n)}{n^{s}}, \end{equation*}
where $\betn(n)$ is multiplicative and supported on $\Sm(X)$, $|\betn(n)|\ll d(n)$, and for $n \in \Sm(\sqrt{X})$ and $p \in \Sm(X)$, we have $\betn(n) = \mu(n)\chi(n)$ and $\betn(p) = \mu(p)\chi(p)$. 

Now, further, define $Q_X(s) = Q_X(s,1)$ where $1$ here is the sole character modulo $1$, and let
\begin{equation*} Q_X(s) = \sum_{n=1}^\infty \frac{\alphn(n)}{n^s}. \end{equation*}
Then we see that $\alphn(n)$ as defined above is the same as in \cite[Section 5]{gonek2007hybrid}, and further it is immediate that for $n \in \Sm(X)$, $\betn(n) = \alphn(n)\chi(n)$. 

Mimicking the argument for (\ref{eqn: pishortdir}), one can show that
\begin{equation}\label{eqn: qshortpoly}\begin{split} Q_X\left(\tfrac{1}{2}+it,\chi\right) &{}= \sum_{\substack{n\leqs T^\theta\\n\in \Sm(X)}} \frac{\betn(n)}{n^{1/2+it}} + \bigO_{\epsilon,\theta}(T^{-\theta\epsilon/10})\\ &{}= \sum_{\substack{n\leqs T^\theta\\n\in \Sm(X)}} \frac{\alphn(n)\chi(n)}{n^{1/2+it}} + \bigO_{\epsilon,\theta}(T^{-\theta\epsilon/10}), \end{split}\end{equation} 
for $\epsilon>0$ small enough.

Putting $\theta = 1/20$, and $b(n) = \alphn(n)$ in Lemma~\ref{lem: twisted2ndmom}, we get that
\begin{equation} M(T;\chi,\alphn) = M'(T;\chi,\alphn) + \bigO(T^{-\varepsilon}), \label{eqn: M=M'}\end{equation}
with $M = M_{\frac{1}{20}}$, $M' = M'_{\frac{1}{20}}$ and $\varepsilon = \varepsilon_{\frac{1}{20}}>0$. 

We first compute the main term $M'(T;\chi,\alphn)$. Since, $[m,n](m,n) = mn$, $M'(\chi,\alphn,T)$ is
\begin{equation*} \frac{\varphi(q)}{q}\sum_{\substack{m,n\leqs T^{1/20}\\m,n\in \Sm(X)}} \frac{\alphn(m)}{m}\frac{\alphn(n)}{n} (m,n) \left\{\log\left(\frac{qT(m,n)^2}{2\pi mn}\right) + \bigO_q(1)\right\}. \end{equation*}

Now, note that any estimates \cite[pp. 530-531] {gonek2007hybrid} can be applied to the above, provided we add the restrictions $(m,q) = (n,q) = (g,q) = 1$ to the sums, and replace $\log T$ with $\log qT$. In particular, following the argument for \cite[Equation~34]{gonek2007hybrid}, we conclude that $M'(T;\chi,\alphn)$ is
\begin{equation*}  \frac{\varphi(q)\log qT}{q}\sum_{\substack{m,n\leqs T^{1/20}\\m,n\in \Sm(X)}} \frac{\alphn(m)}{m}\frac{\alphn(n)}{n}(m,n) + \bigO_q((\log X)^{10}). \end{equation*}
Since $\sum_{\substack{g \mid m \\ g \mid n}} \varphi(g) = (m,n)$, the inner sum is 
\begin{equation*} \sum_{\substack{m,n\leqs T^{1/20}\\m,n\in \Sm(X)}} \frac{\alphn(m)}{m}\frac{\alphn(n)}{n}\sum_{\substack{g \mid m \\ g \mid n}} \varphi(g) = \sum_{\substack{g\leqs T^{1/20} \\ g \in \Sm(X)}} \frac{\varphi(g)}{g^2}\Bigg(\sum_{\substack{n\leqs \frac{T^{1/20}}{g}\\n \in \Sm(X))}} \frac{\alphn(gn)}{n}\Bigg)^2 \end{equation*}
Following the argument for \cite[Equation~37]{gonek2007hybrid} here, we can extend both the summations above to infinity to get that $M'(T;\chi,\alphn)$ is
\begin{equation*} \frac{\varphi(q)\log qT}{q}\sum_{\substack{g \in \Sm(X)}} \frac{\varphi(g)}{g^2}\left(\sum_{\substack{n \in \Sm(X))}} \frac{\alphn(gn)}{n}\right)^2 + \bigO_q((\log X)^{10}). \end{equation*}
By the muliplicativity of $\alphn$ and $\varphi$, the sum here can be written as an Euler product
\begin{equation*} \prod_{\substack{p \leqs X\\p \nmid q}} \left(\sum_{r,j,k\geqs 0} \frac{\varphi(p^r)\alphn(p^{r+j})\alphn(p^{r+k})}{p^{2r+j+k}}\right).\end{equation*}

Now, recalling that $\alphn(n) = \mu(n)$ if $n \in \Sm(\sqrt{X})$, $\alphn(p) = \mu(p)$ for all $p \leqs X$ and $\alphn(n) \ll d(n)$ for all $n \in \Sm(X)$, we get that this product is equal to
\begin{equation*} \begin{split}\prod_{\substack{p\leqs \sqrt{X}\\p \nmid q}}&\left(1 - \frac{1}{p}\right)\prod_{\substack{\sqrt{X}<p\leqs X\\p\nmid q}}\left(1 - \frac{1}{p} +\bigO\left(\frac{1}{p^2}\right)\right) \\&{}= \frac{q}{\varphi(q)} \prod_{p\leqs X}\left(1 - \frac{1}{p}\right)\prod_{\sqrt{X}<p\leqs X}\left(1 + \bigO\left(\frac{1}{p^2}\right)\right)\\&{} = \frac{q}{\varphi(q)}\cdot\frac{1}{e^\gamma\log X}\left(1 + \bigO\left(\frac{1}{\log X}\right)\right).\end{split} \end{equation*}
Thus, since $\log X \ll \log\log T$, we see that, in fact
\begin{equation} \label{eqn: maintermsplitting} M'(T;\chi,\alphn) = \frac{\log qT}{e^\gamma\log X}\left(1 + \bigO\left(\frac{1}{\log X}\right)\right) \end{equation} 

Writing (\ref{eqn: qshortpoly}) with $\theta = 1/20$ as $Q_X(1/2+it,\chi) = \QS + \bigO(T^{-\epsilon/200})$,

\begin{multline*} \frac{1}{T}\int_T^{2T}\left|L\left(\tfrac{1}{2}+it,\chi\right)Q_X\left(\tfrac{1}{2}+it,\chi\right)\right|^2 \,dt \\{} = \frac{1}{T}\int_T^{2T}\left|L\left(\tfrac{1}{2}+it,\chi\right)\QS\right|^2\,dt \\{}+\bigO\left(\frac{1}{T^{1+\tfrac{\epsilon}{200}}}\int_T^{2T}\left|L\left(\tfrac{1}{2}+it,\chi\right)^2\QS\right|\,dt\right) \\{}+ \bigO\left(\frac{1}{T^{1+\tfrac{\epsilon}{100}}}\int_T^{2T}\left|L\left(\tfrac{1}{2}+it,\chi\right)\right|^2\,dt\right). \end{multline*}
The first term here is $M'(T;\chi,\alphn) + \bigO(T^{-\varepsilon})$. The last term is $\ll_q T^{-\epsilon/200}$ since the second moment of $L(s,\chi)$ is $\ll_q T\log T$. Finally, by Cauchy-Schwarz and (\ref{eqn: maintermsplitting}), the second term is

\begin{equation*} \begin{split} &{}\ll \frac{1}{T^{1+\tfrac{\epsilon}{200}}}\left(\int_T^{2T}\left|L\left(\tfrac{1}{2}+it,\chi\right)\QS\right|^2\,dt\int_T^{2T}\left|L\left(\tfrac{1}{2}+it,\chi\right)\right|^2\,dt\right)^{1/2} \\&{}\ll \frac{1}{T^{1+\tfrac{\epsilon}{200}}}\left(\frac{T^2(\log T)^2}{\log X}\right)^{1/2} \ll T^{-\tfrac{\epsilon}{400}}. \end{split} \end{equation*}

Putting these estimates together with (\ref{eqn: maintermsplitting}), we get that
\begin{equation*} \begin{split} \frac{1}{T}\int_T^{2T}&\left|L\left(\tfrac{1}{2}+it,\chi\right)Q_X\left(\tfrac{1}{2}+it,\chi\right)\right|^2 \,dt \\&{}= M'(T;\chi,\alphn) + \bigO(T^{-\vartheta})\\&{} = \frac{\log qT}{e^\gamma \log X}\left(1 + \bigO\left(\frac{1}{\log X}\right)\right), \end{split} \end{equation*} 
for some $\vartheta = \vartheta(\epsilon,\varepsilon_{\frac{1}{20}}) > 0$ completing the proof of Theorem~\ref{thm: smallk}.
 
\section{Evidence for Conjecture~\ref{conj: hurwitzmoments}} \label{sec: hurwitzmoments}

In this section, we will discuss the proofs of Theorem~\ref{thm: rationalfourthmoment}, Theorem~\ref{thm: L-asymptotic}, Theorem~\ref{thm: constant} and Theorem~\ref{thm: upperandlower}. We will also show that in all cases where our conjectures are known, the constants match up with our predictions. This will complete the presentation of our evidence for Conjecture~\ref{conj: hurwitzmoments}.

In several results here, we must assume one of the following two hypotheses:
\begin{itemize}
\item The Generalized Riemann Hypothesis holds for $L(s,\chi)$ for every character $\chi$ modulo $q$. We denote this by $\GRH$.
\item The mean squares of $\LL(s)$ on the critical line for all $|\vg{\ell}| = k$ satisfy Conjecture~\ref{conj: splitting} and Conjecture~\ref{conj: randommatrices}. We denote this by $\Sp$, for splitting conjecture.
\end{itemize}
We introduce the above shorthand for convenience, as many of the results in this section will hold under either hypothesis.

\subsection{The Asymptotic Formula for Moments of Products of Dirichlet \texorpdfstring{$L$}{L}-functions}
In this subsection, we will prove Theorem~\ref{thm: L-asymptotic}. This is straightforward. By Theorem~\ref{thm: eulerproductmoment}, we get that assuming we hold $q, \vg{\ell},$ and $\epsilon$ fixed, and let $X, T \to \infty$ with $X \ll_\epsilon (\log T)^{2-\epsilon}$,
\begin{equation} \label{eqn: plmoment} \frac{1}{T}\int_T^{2T} |\PP(s)|^2\,dt = (e^\gamma\log X)^\lambda \prod_{p}\left\{\left(1 - \frac{1}{p}\right)^{\lambda} \sum_{m=0}^\infty \frac{|\dl(p^m)|^2}{p^m} \right\}.\end{equation}
Further, since we are assuming Conjecture~\ref{conj: randommatrices} for $\vg{\ell}$, we get that under the same conditions as before, 
\begin{equation} \label{eqn: zlmoment}\begin{split} \frac{1}{T}\int_T^{2T} |\ZZ(s)|^{2}\,dt &{}\sim \prod_\chi \left[\frac{G(\ell_\chi+1)^2}{G(2\ell_\chi +1)} \left(\frac{\log \qchi T}{e^\gamma \log X}\right)^{\ell_\chi^2}\right]\\& = \frac{1}{(e^\gamma\log X)^\lambda} \prod_\chi\left[ \frac{G(\ell_\chi+1)^2}{G(2\ell_\chi +1)} \left(\log \qchi T\right)^{\ell_\chi^2}\right].\end{split}\end{equation}
Finally, since we are assuming that Conjecture~\ref{conj: splitting} is true for $\vg{\ell}$, we get that for $X, T$ as before,
\begin{equation*} \frac{1}{T}\int_T^{2T}\left|\LL(s)\right|^2\,dt \sim \left(\frac{1}{T}\int_T^{2T}\left|\PP(s)\right|^2\,dt\right) \times \left(\frac{1}{T}\int_T^{2T} \left|\ZZ(s)\right|^2\,dt \right). \end{equation*}
Multiplying (\ref{eqn: plmoment}) and (\ref{eqn: zlmoment}) and inserting above, we see that the $(e^\gamma \log X)^\lambda$ factors cancel out, and the constants combine to become $\clq$, giving
\begin{equation*} \frac{1}{T}\int_T^{2T}\left|\LL(s)\right|^2\,dt \sim \clq \prod_\chi \bigg(\log \qchi T\bigg)^{\ell_\chi^2}, \end{equation*}
as desired. 

\subsection{Computing \texorpdfstring{$c_k(\alpha)$}{ck(a)}} \label{subs: computing}

The main result of this subsection is the following proposition which is one way to make the heuristic in (\ref{eqn: heuristicanswer}) rigorous:

\begin{proposition}

Let $M_k(T;\alpha)$ be as in (\ref{def:Mkalpha}), and for any Dirichlet character $\chi$ modulo $q$, define $M_k(T;\chi)$ by
\begin{equation*} \label{prop: mainprop} M_k(T;\chi) = \int_{T}^{2T} \left|L\left(\tfrac{1}{2}+it,\chi\right)\right|^{2k}\,dt. \end{equation*}
If either $\GRH$ or $\Sp$ holds, and $\alpha = a/q$ with $(a,q) = 1$, then
\begin{equation*} M_k(T;\alpha) = \frac{q^k}{\varphi(q)^{2k}} \sum_\chi M_k(T;\chi) + o_{q,k}(T(\log T)^{k^2}). \end{equation*}
If $k = 1$ or $k = 2$, then the above can be proved unconditionally.
\end{proposition}

We show that this proposition establishes Theorem~\ref{thm: constant}. Note that under the hypothesis of Theorem~\ref{thm: constant}, $\Sp$ holds, and hence so does the conclusion of Theorem~\ref{thm: L-asymptotic}. Thus, for a fixed $q,\chi$, and with $\vg{\ell} = k\vg{\delta}^\chi$, we get the asymptotic \begin{equation*} M_k(T;\chi) = (\clq +o_{q,k}(1)) T(\log T)^{k^2}. \end{equation*}
Thus, by Proposition~\ref{prop: mainprop},
\begin{equation*} M_k(T;\alpha) = \frac{q^k}{\varphi(q)^{2k}} \left(\sum_{\vg{\ell} = k\vg{\delta}^\chi} \clq\right) T(\log T)^{k^2} + o_{q,k}(T(\log T)^{k^2}), \end{equation*}
which establishes Theorem~\ref{thm: constant} with 
\begin{equation} c_k(\alpha) =  \frac{q^k}{\varphi(q)^{2k}} \left(\sum_{\vg{\ell} = k\vg{\delta}^\chi} \clq\right), \label{eqn: ckainchi}\end{equation}
where the sum runs over all tuples $\vg{\ell}$ of the form $k\vg{\delta}^\chi$ for some character $\chi$. It remains to simplify the constant. Note that for $\vg{\ell} = k\vg{\delta}^\chi$, $\dl(n) = \chi(n)d_k(n)$, where $d_k(n)$ is the usual divisor function. In particular, this means that $|\dl(p^m)|^2 = \chi_0(p^m)d_k(p^m)^2$, and hence $\clq$ depends only on the modulus of $\chi$. Further, $\lambda(\vg{\ell}) = k^2$. Thus,
\begin{equation*} \clq = \prod_{p}\left\{\left(1 - \frac{1}{p}\right)^{k^2}\sum_{m=0}^\infty \frac{\chi_0(p^m)d_k(p^m)^2}{p^m}\right\} \frac{G(k+1)}{G(2k +1)}, \end{equation*}
for every $\vg{\ell}$ appearing in the sum in (\ref{eqn: ckainchi}). We see that $\clq$ is the same as usual constant for $\zeta(s)$, $c_k = c_k(1)$ but with a slight change in the local factors in the Euler product corresponding to those primes $p$ which divide $q$. That is, 
\begin{equation} \begin{split} \clq = c_k \prod_{p\mid q}\left\{\sum_{m=0}^\infty \frac{d(p^m)^2}{p^m}\right\}^{-1}= c_k \prod_{p \mid q}\left\{\sum_{m=0}^\infty \binom{m+k-1}{k-1}^2 p^{-m}\right\}^{-1}. \end{split} \label{eqn: constantforchi}\end{equation}
Substituting this back into (\ref{eqn: ckainchi}),
\begin{equation*} \begin{split} c_k(\alpha) &{}= c_k \frac{q^k}{\varphi(q)^{2k-1}} \prod_{p\mid q} \left\{\sum_{m=0}^\infty \binom{m+k-1}{k-1}^2 p^{-m}\right\}^{-1}, \end{split} \end{equation*}
as desired. This completes the proof of Theorem~\ref{thm: constant} from Proposition~\ref{prop: mainprop}.

We now turn to the proof of Proposition~\ref{prop: mainprop}. By (\ref{eqn: hurwitzdecomp}), $M_k(T;\alpha)$ is equal to
\begin{equation} \label{eqn: sumoverl} \frac{q^{k}}{\varphi(q)^{2k}} \sum_{\substack{|\vg{\ell}^{(1)}| = k,\\|\vg{\ell}^{(2)}|=k}}\binom{k}{\vg{\ell}^{(1)}} \binom{k}{\vg{\ell}^{(2)}} \left[\prod_{\chi} \chi(a)^{\ell^{(2)}_{\chi} - \ell^{(1)}_{\chi}}\right] \int_T^{2T}\LLa(s)\ov{\LLb(s)}\,dt, \end{equation}
where $s = 1/2 + it$. We divide the terms in the sum into four types:
\begin{itemize}
\item The primary diagonal terms. These correspond to $\vg{\ell}^{(1)} = \vg{\ell}^{(2)} = k\vg{\delta}^{\chi}$ for some character $\chi$. For such terms, it is clear that $\binom{k}{\vg{\ell}^{(j)}} = 1$ and the integral devolves to $M_k(T;\chi)$. 
\item The secondary diagonal terms. These correspond to diagonal terms $\vg{\ell}^{(1)} = \vg{\ell}^{(2)}$ which are not main diagonal terms. Thus, $\vg{\ell} = \vg{\ell}^{(1)} = \vg{\ell}^{(2)} \neq k\vg{\delta}^\chi$ for every character $\chi$ modulo $q$. For such terms, the integral devolves to the mean square of $\LL(1/2+it)$ over $[T,2T]$. 
\item The major off-diagonal terms. These correspond to $\vg{\ell}^{(1)} = k\vg{\delta}^{\chi}$ and $\vg{\ell}^{(2)} = k\vg{\delta}^\nu$ for distinct characters $\chi,\nu$. For these terms, the integral devolves to $\int_T^{2T} L(s,\chi)^k \ov{L(s,\nu)}^k \,dt$. 
\item The minor off-diagonal terms. These correspond to any terms which are not of any of the above three forms. 
\end{itemize}

The primary diagonal terms clearly give rise to the  main term in Proposition~\ref{prop: mainprop}. We will show, through a series of lemmata, that all the other terms can be subsumed by the error term, thereby proving the proposition. 

The following lemma is a corollary of a result of Milinovich and Turnage-Butterbaugh \cite{milinovich2014moments}:

\begin{lemma} \label{lem: L-upperbound}

Suppose that either $\GRH$ or $\Sp$ holds, and that $\vg{\ell}$ is tuple of nonnegative integers indexed by the characters modulo $q$ satisfying $|\vg{\ell}| = k$. Then, for $\lambda(\vg{\ell}) = \sum_\chi \ell_\chi^2$ and any $\epsilon>0$, 
\begin{equation*} \int_T^{2T} \left|\LL\left(\tfrac{1}{2}+it\right)\right|^2 \,dt \ll_{q,k,\epsilon} T(\log T)^{\lambda + \epsilon}. \end{equation*}
In particular, if $\vg{\ell} \neq k \vg{\delta}^\chi$ for all characters $\chi$ modulo $q$, then
\begin{equation*} \int_T^{2T} \left|\LL\left(\tfrac{1}{2}+it\right)\right|^2 \,dt \ll_{q,k,\epsilon} T(\log T)^{k^2 - 1 + \epsilon}, \end{equation*}
and hence, the secondary diagonal terms in (\ref{eqn: sumoverl}) contribute at most $o_{q,k}(T(\log T)^{k^2})$ to the sum. 
\end{lemma}

\begin{proof}
First, suppose that $\Sp$ holds. Then, the first inequality is trivally true due to Theorem~\ref{thm: L-asymptotic}. Alternatively, suppose that $\GRH$ holds. Then, the first inequality follows by applying \cite[Theorem 1.1]{milinovich2014moments} in the specific case where all the $L$-functions involved are Dirichlet $L$-functions. 

Now note that under the constraints $\ell_\chi \geqs 0$ and $\sum_\chi \ell_\chi = k$, we have that $\lambda = \sum_\chi \ell_\chi^2 \leqs k^2$ with equality if and only if the entire weight of $\vg{\ell}$ is concentrated on a single character. In particular, if $\vg{\ell} \neq k \vg{\delta}^\chi$ for all characters $\chi$, then $\lambda(\vg{\ell}) < k^2$ and so, $\lambda(\vg{\ell}) \leqs k^2 - 1$. Thus, the second inequality in the lemma follows from the first. 

Finally, it is clear that the constant $\binom{k}{\vg{\ell}}^2$ in the the secondary diagonal terms can be subsumed into the implicit constant from the second bound. Taking, for example, $\epsilon = 1/2$ shows that each such term is $o_{q,k}(T(\log T)^{k^2})$ and since there are only $\ll_{q,k} 1$ such terms it follows that these contribute only $o_{q,k}(T(\log T)^{k^2})$.

\end{proof}

The minor off-diagonal terms can be handled by Cauchy-Schwarz. This is very far from sharp, but our other error terms are already of size $\asymp_{q,k} T(\log T)^{k^2 -2k + 2}$ and so this suffices for our purposes.

\begin{lemma} \label{lem: minor} 

Suppose that either $\GRH$ or $\Sp$ holds and that $\vg{\ell}^{(1)}$ and $\vg{\ell}^{(2)}$ are tuples of nonnegative integers characters modulo $q$ satisfying $|\vg{\ell}^{(1)}|=|\vg{\ell}^{(2)}| = k$. Further, suppose that $\vg{\ell}^{(1)},\vg{\ell}^{(2)}$ correspond to a minor off-diagonal term, as defined above. Then, for $s = 1/2 + it$ and any $\epsilon>0$,
\begin{equation*} \int_T^{2T} \LLa(s)\ov{\LLb(s)} \,dt \ll_{q,k,\epsilon} T(\log T)^{k^2 -1/2 + \epsilon}, \end{equation*}
and hence, the minor off-diagonal terms in (\ref{eqn: sumoverl}) contribute at most $o_{q,k}(T(\log T)^{k^2})$ to the sum.
\end{lemma}
\begin{proof} Since $(\vg{\ell}^{(1)},\vg{\ell}^{(2)})$ corresponds to an off-diagonal term, $\vg{\ell}^{(1)} \neq \vg{\ell}^{(2)}$. Further, since it is not a major off-diagonal term we must have that either $\vg{\ell}^{(1)}$ or $\vg{\ell}^{(2)}$ is not of the form $k \vg{\delta}^\chi$ for some character $\chi$. Due to symmetry, we can assume without loss of generality that $\vg{\ell}^{(1)} \neq k \vg{\delta}^\chi$ for all characters $\chi$. Then, by Cauchy-Schwarz and Lemma~\ref{lem: L-upperbound}, we get for $s = 1/2 + it$,

\begin{equation*} \begin{split} \int_T^{2T} \LLa(s)&\ov{\LLb(s)} \,dt \\&{}  \ll \left(\int_T^{2T} \left|\LLa(s)\right|^2 \,dt\right)^{1/2} \left(\int_T^{2T} \left|\LLb(s)\right|^2 \,dt\right)^{1/2} \\{}& \ll_{q,k,\epsilon} \left\{T(\log T)^{k^2 - 1 + \epsilon}\right\}^{1/2} \left\{T(\log T)^{k^2+\epsilon}\right\}^{1/2} \\&{}= T(\log T)^{k^2 - 1/2 + \epsilon}. \end{split}\end{equation*}

Showing that these terms contribute to the error is similar to the previous lemma, hence we omit the proof.

\end{proof}

We note in passing that if $k = 1$, there are no secondary diagonal terms or minor off-diagonal terms, and so the previous two lemmata are unnecessary. 

It remains to deal with the major non-diagonal terms. If $k \geqs 2$, then again Cauchy-Schwarz suffices. 

\begin{lemma} \label{lem: major}

Suppose that either $\GRH$ or $\Sp$ holds for some $k \geqs 2$ and that $\chi$ and $\nu$ are distinct characters modulo $q$. Then, for $s = 1/2 + it$ and any $\epsilon>0$,
\begin{equation*} \int_T^{2T} L(s,\chi)^k \ov{L(s,\nu)}^k \,dt \ll_{q,k,\epsilon} T(\log T)^{k^2 - 2k + 2 + \epsilon}, \end{equation*}

and hence, the major off-diagonal terms in (\ref{eqn: sumoverl}) contribute at most $o_{q,k}(T(\log T)^{k^2})$ to the sum.
\end{lemma}

\begin{proof} Note that,
\begin{equation*} L(s,\chi)^k \ov{L(s,\nu)}^k = \left[L(s,\chi)^{k-1} \ov{L(s,\nu)}\right]\left[L(s,\chi) \ov{L(s,\nu)}^{k-1}\right]. \end{equation*}
Thus, setting 
\begin{equation*} \vg{\ell}^{(1)} = (k-1)\vg{\delta}^\chi + \vg{\delta}^\nu \end{equation*} and \begin{equation*} \vg{\ell}^{(2)} = \vg{\delta}^\chi + (k-1)\vg{\delta}^\nu, \end{equation*}
we see by Cauchy-Schwarz and Lemma~\ref{lem: L-upperbound} that
\begin{equation*} \begin{split} \int_T^{2T} &L(s,\chi)^k \ov{L(s,\nu})^k \,dt \\{} & \ll \left(\int_T^{2T} \left|\LLa(1/2+it)\right|^2 \,dt\right)^{1/2} \left(\int_T^{2T} \left|\LLb(1/2+it)\right|^2 \,dt\right)^{1/2} \\{}& \ll_{q,k,\epsilon} \left\{T(\log T)^{k^2 - 2k +2 + \epsilon}\right\}^{1/2} \left\{T(\log T)^{k^2 -2k+2+\epsilon}\right\}^{1/2} \\&{}= T(\log T)^{k^2 - 2k+2 + \epsilon}, \end{split}\end{equation*}
proving the desired bound. Showing that these terms contribute to the error is similar to the above lemmata, and hence omitted.

\end{proof}

From the discussion above, it remains to deal with the off-diagonal terms when $k=1$, and to show that the argument can be made unconditional for $k = 2$. We postpone the latter to Section~\ref{subs: rationalfourthmoment}, as it will be a corollary of the discussion about Theorem~\ref{thm: rationalfourthmoment}. 

For the former, since we also claimed that Proposition~\ref{prop: mainprop} is unconditional in this case, we cannot use the hypotheses $\GRH$ or $\Sp$. For such terms, standard techniques developed to handle the mean square of $\zeta(s)$ can be applied. For our purposes, the following lemma suffices:

\begin{lemma} Let $\chi$ and $\nu$ be distinct characters modulo $q$. Then, for $s = 1/2 + it$,

\begin{equation*} \int_T^{2T} L(s,\chi)\ov{L(s,\nu)} \,dt \ll_q T(\log T)^{3/4}, \end{equation*}
unconditionally. Hence, if $k = 1$, the off-diagonal terms in (\ref{eqn: sumoverl}) contribute only $o_q(T\log T)$ to the sum.

\end{lemma}
\begin{proof}

The upper bound is \cite[Equation 4]{ishikawa2006difference} with $\chi_j = \chi$ and $\chi_k = \nu$. Showing that these terms contribute to the error is similar to the previous lemmata, and hence omitted.

\end{proof}

Proposition~\ref{prop: mainprop} follows by putting all these lemmata together, thus completing the proof of Theorem~\ref{thm: constant}.

\subsection{Upper and Lower Bounds for \texorpdfstring{$M_k(T;\alpha)$}{Mk(T)}} In order to prove Theorem~\ref{thm: upperandlower}, we have to find bounds on $M_k(T;\alpha)$ conditionally on GRH. 

The claimed upper bound follows trivially from the previous subsection, since Proposition~\ref{prop: mainprop} tells us that on $\GRH$,
\begin{equation*} M_k(T;\alpha) \ll_{q,k} \sum_\chi M_k(T;\chi) + T(\log T)^{k^2}, \end{equation*}
and Lemma~\ref{lem: L-upperbound} tells us that on $\GRH$,
\begin{equation*} M_k(T;\chi) \ll_{q,k,\epsilon} T(\log T)^{k^2 + \epsilon}. \end{equation*}

To prove the lower bound, we proceed by reducing the problem to computing lower bounds for the moments of $\zeta(s)$, i.e. lower bounds on $M_k(T)$. The key fact is the following obvious lemma:

\begin{lemma} Let $\chi_0$ be the principal Dirichlet  character modulo $q$. Then, 
\begin{equation*} \int_T^{2T} \left|L\left(\tfrac{1}{2}+it,\chi_0\right)\right|^{2k}\,dt \asymp_{q,k} \int_T^{2T} \left|\zeta\left(\tfrac{1}{2}+it\right)\right|^{2k}\,dt. \end{equation*}

\end{lemma}

In particular, this tells us that $M_k(T;\chi_0) \gg_{q,k} M_k(T)$. By the deep results in the literature about lower bounds for $M_k(T)$ mentioned in the introduction, we can conclude that in fact $M_k(T;\chi_0) \gg_{q,k} T(\log T)^{k^2}$. 

Then, by Proposition~\ref{prop: mainprop}, we have conditionally on GRH,
\begin{equation*} \begin{split} M_k(T;\alpha) &{}\gg_{q,k} \sum_{\chi} M_k(T;\chi) + o_{q,k}(T (\log T)^{k^2})\\&{}\geqs M_k(T;\chi_0) + o_{q,k}(T(\log T)^{k^2}) \\&{} \gg_{q,k} T(\log T)^{k^2}, \end{split} \end{equation*}
completing the proof.

\subsection{The Fourth Moment of \texorpdfstring{$\zeta(s,\alpha)$}{the Hurwitz Zeta function}} \label{subs: rationalfourthmoment}
The goal here is to compute the asymptotic for $M_2(T;\alpha)$, for $\alpha \in \Q$ originally proved (unpublished) in Andersson's thesis \cite[pp. 71-72]{anderssonthesis}. We reprove this here as, in the process, we will be able to verify that our conjectures for the constants $\clq$ and $c_k(\alpha)$ are correct when $|\vg{\ell}| \leqs 2$ or $k \leqs 2$. Further, our discussion will imply that the conclusion in Proposition~\ref{prop: mainprop} is true unconditionally if $k = 2$.

To do this, we make use of a recent result of Topacogullari \cite{topacogullari2020fourth}, where he computes the full asymptotic formula for the fourth moments of $L(s,\chi)$ and the mean-square of $L(s,\chi)L(s,\nu)$ with a power saving in the error term, and an explicit dependence on the conductors. We state the weaker result we need as propositions.

\begin{proposition} \label{prop: Lfourth}Let $\chi$ be a Dirichlet character modulo $q$. Then, for $s = 1/2+it$,
\begin{equation*} \int_T^{2T} \left|L\left(s,\chi\right)\right|^4\,dt = C(\chi) T(\log T)^4 + O_q(T(\log T)^3) \end{equation*}
where $C(\chi)$ is given by
\begin{equation*} C(\chi) = \frac{1}{2\pi^2}\frac{\varphi(q)^2}{q^2} \prod_{p \mid q}\left(1 - \frac{2}{p+1}\right). \end{equation*}
\end{proposition}
\begin{proof} This is an immediate corollary of \cite[Theorem~1.1]{topacogullari2020fourth}. 
\end{proof}

\begin{proposition} \label{prop: Lmean} Let $\chi$ and $\nu$ be distinct Dirichlet characters modulo $q$. Then, for $s = 1/2 + it$,
\begin{equation*} \int_T^{2T} \left|L\left(s,\chi\right)L\left(s,\nu\right)\right|^2\,dt = D(\chi,\nu) T(\log T)^2 + O_q(T\log T) \end{equation*}
where $D(\chi,\nu)$ is given by
\begin{equation*} D(\chi,\nu) = \frac{6}{\pi^2}|L(1,\chi\ov{\nu})|^2 \frac{\varphi(q)}{q} \prod_{p \mid q}\left(1 - \frac{1}{p+1}\right). \end{equation*}
\end{proposition}

\begin{proof} This is a corollary of \cite[Theorem~1.3]{topacogullari2020fourth}, by setting $\chi_1 = \chi$, $\chi_2 = \nu$, $q_1 = q_2 = q$, noting that this implies $q_1^\star = q_2^\star = 1$ and noting that $\varphi(q^2) = q\varphi(q)$. 
\end{proof}

We note here that the previous two propositions show that the result of Proposition~\ref{prop: mainprop} can be obtained unconditionally when $k=2$, which we had not shown previously. This is because the hypotheses $\GRH$ or $\Sp$ were used only in the proof of Lemma~\ref{lem: L-upperbound} and this use can be replaced by the above propositions, which trivially give the bound 
\begin{equation*} \int_T^{2T} \left|\LL\left(\tfrac{1}{2}+it\right)\right|^2 \,dt \ll_q T(\log T)^\lambda \end{equation*} 
for $\vg{\ell}$ satisfying $|\vg{\ell}| = 2$. 

Now, these propositions clearly raise the question of whether the constants in them are consistent with the conjectural constant one obtains in Theorem~\ref{thm: L-asymptotic} with $|\vg{\ell}| = 2$. Let $C'(\chi)$ and $D'(\chi,\nu)$ be the constants predicted by Theorem~\ref{thm: L-asymptotic}. Then, $C'(\chi) = \clq$ for $\vg{\ell} = 2\vg{\delta}^{\chi}$, and $D'(\chi,\nu) = \clq$ for $\vg{\ell} = \vg{\delta}^\chi + \vg{\delta}^\nu$, $\chi \neq \nu$. 

To show that $C(\chi) = C'(\chi)$ and $D(\chi,\nu) = D'(\chi,\nu)$, the plan of attack will be to write everything involved as an Euler product, and then compare what happens on both sides in the local factors for different primes $p$. 

In particular, recall Ingham's result that $c_2 = \frac{1}{2\pi^2}$. Thus, using this, we can suppress the local factors for $p \nmid q$ when showing $C(\chi) = C'(\chi)$. Rewriting $C(\chi)$ in Euler product form using a standard formula for $\varphi(q)/q$, we see that
\begin{equation} \label{eqn: cchi}\begin{split} C(\chi) &{}= c_2 \prod_{p\mid q}\left(1 - \frac{1}{p}\right)^2\left(1 - \frac{2}{p+1}\right)\\&{}=c_2 \prod_{p\mid q}\left(1 - \frac{1}{p}\right)^3\left(1 + \frac{1}{p}\right)^{-1}. \end{split} \end{equation}
Now since $C'(\chi) = \clq$ for $\vg{\ell} = 2\vg{\delta}^\chi$, we get that $\lambda = 2^2 = 4$, $\dl(n) = \chi_0(n)d_2(n)$ where $\chi_0$ is the principal character modulo $q$, and hence
\begin{equation*}\begin{split} C'(\chi) = \clq = \left[\prod_{p}\left\{\left(1 -\frac{1}{p}\right)^4\sum_{m=0}^\infty \frac{\chi_0(p^m)d_2(p^m)^2}{p^m}\right\}\right] \left[\frac{G(3)^2}{G(5)}\right]. \end{split}\end{equation*}
Recall that
\begin{equation*} c_2 = \left[\prod_{p}\left\{\left(1 -\frac{1}{p}\right)^4\sum_{m=0}^\infty \frac{d_2(p^m)^2}{p^m}\right\}\right] \left[\frac{G(3)^2}{G(5)}\right]. \end{equation*}
Thus, we see that
\begin{equation} \label{eqn: c'chi} C'(\chi) = c_2 \prod_{p \mid q} \left\{\sum_{m=0}^\infty \frac{d_2(p^m)^2}{p^m}\right\}^{-1}. \end{equation}

In light of (\ref{eqn: cchi}) and (\ref{eqn: c'chi}), it suffices to note the power series equality 
\begin{equation*} \sum_{m=0}^\infty d_2(p^m)^2 z^m = \frac{1 + z}{(1 - z)^3}, \end{equation*}
for $|z|<1$, as then plugging in $z = 1/p$ and taking products over $p \mid q$ gives us $C(\chi) = C'(\chi)$. To see the above power series equality, note that $d_2(p^m) = m+1$ and hence this follows straightforwardly from the geometric series formula.

Now, we turn to showing $D(\chi,\nu) = D'(\chi,\nu)$. Since $D'(\chi,\nu) = \clq$ for $\vg{\ell} = \vg{\delta}^\chi + \vg{\delta}^\nu$, hence $\LL(s) = L(s,\chi)L(s,\nu)$ and $\dl = \chi*\nu$, where $*$ denotes Dirichlet convolution. Because $\nu$ is completely multiplicative, $\dl(n) = \nu(n) \{1*(\chi\ov{\nu})\}(n)$. In particular, it follows that $|\dl(n)|^2$ depends only on $\chi\ov{\nu}$ and not the individual characters $\chi$ and $\nu$. Thus, $D'(\chi,\nu)$ also depends only on $\chi\ov{\nu}$. By inspection, we see that $D(\chi,\nu)$ also depends only on $\chi\ov{\nu}$. Thus, without loss of generality, we can assume that $\nu = \chi_0$.  It now suffices to show that for $\chi \neq \chi_0$, $D(\chi,\chi_0) = D'(\chi,\chi_0)$.

For $\vg{\ell} = \vg{\delta}^\chi + \vg{\delta}^{\chi_0}$, we see that $\lambda = 1^2 + 1^2 = 2$. Further, $\dl(n) = \chi_0(n) \{1*\chi\}(n)$. Finally the product over $\chi$ in the expression for $\clq$ vanishes, since $G(1)^2/G(3) = 1$. Thus, we get
\begin{equation} \label{eqn: d'chi}D'(\chi,\chi_0) = \prod_p \left\{\left(1 - \frac{1}{p}\right)^2 \sum_{m=0}^\infty \frac{\chi_0(p^m)|(1*\chi)(p^m)|^2}{p^m}\right\} \end{equation}

Now, using the Euler product formulae,
\begin{equation*} \frac{6}{\pi^2} = \frac{1}{\zeta(2)} = \prod_{p} \left(1 - \frac{1}{p^2}\right) \end{equation*}
and
\begin{equation*} L(1,\chi) = \prod_{p} \frac{1}{1 - \chi(p)p^{-1}}, \end{equation*}
where the latter holds because $\chi \neq \chi_0$, we see that
%\begin{equation} \label{eqn: dchi}\begin{split} D(\chi,\chi_0)&{} = \frac{6}{\pi^2} |L(1,\chi)|^2 \frac{\varphi(q)}{q}\prod_p \left(1 - \frac{1}{p+1}\right) \\ &{} = \frac{6}{\pi^2} L(1,\chi)L(1,\ov{\chi}) \frac{\varphi(q)}{q}\prod_{p\mid q} \left(1 - \frac{1}{p+1}\right) \\&{}=\frac{6}{\pi^2} L(1,\chi)L(1,\ov{\chi}) \frac{\varphi(q)}{q}\prod_{p\mid q} \left(\frac{1}{1+p^{-1}}\right) \\&{}= \left\{\prod_p \frac{1 - p^{-2}}{(1 - \chi(p)p^{-1})(1 - \ov{\chi}(p)p^{-1})}\right\} \left\{ \prod_{p\mid q} \frac{1-p^{-1}}{1+p^{-1}}\right\}. \end{split} \end{equation}
\begin{equation} \label{eqn: dchi}\begin{split} D(\chi,\chi_0)&{} = \frac{6}{\pi^2} |L(1,\chi)|^2 \frac{\varphi(q)}{q}\prod_p \left(1 - \frac{1}{p+1}\right)  \\&{}= \left\{\prod_p \frac{1 - p^{-2}}{(1 - \chi(p)p^{-1})(1 - \ov{\chi}(p)p^{-1})}\right\} \left\{ \prod_{p\mid q} \frac{1-p^{-1}}{1+p^{-1}}\right\}. \end{split} \end{equation}

Comparing the local factors corresponding to primes $p$ dividing $q$, we see that for $D'(\chi,\chi_0)$ these are $(1 - p^{-1})^2$, while for $D(\chi,\chi_0)$, they are
\begin{equation*} \frac{(1 - p^{-2})(1 - p^{-1})}{1 + p^{-1}} = (1 - p^{-1})^2. \end{equation*}
Thus, it remains to check the local factors corresponding to primes $p$ which are coprime to $q$. For $D(\chi,\chi_0)$, these are of the shape
\begin{equation*} \frac{1 - p^{-2}}{(1-\chi(p)p^{-1})(1-\ov{\chi}(p)p^{-1})}, \end{equation*}
while for $D'(\chi,\chi_0)$, these are of the shape
\begin{equation*} \left(1 - \frac{1}{p}\right)^2 \sum_{m=0}^\infty \frac{|(1*\chi)(p^m)|^2}{p^m}. \end{equation*}

Thus, to prove $D(\chi,\chi_0) = D'(\chi,\chi_0)$ it clearly suffices to prove the power series equality
\begin{equation*} \frac{1+z}{(1 - \omega z)(1 - \ov{\omega} z)} = (1 - z) \sum_{m=0}^\infty \left|\sum_{j=0}^m \omega^j\right|^2 z^m, \end{equation*}
for $|z| < 1$ and $|\omega| = 1$, as then plugging in $z = 1/p$, $\omega = \chi(p)$ and multiplying both sides by $(1 - p^{-1})$ gives us the desired equality. 

To prove this power series equality note that both sides are equal to 
\begin{equation} \label{eqn: powerseries} \sum_{m\geqs 0} \left(\sum_{|j|\leqs m} \omega^j\right) z^m, \end{equation}
where the sum over $j$ runs through all integers in $[-m,m]$. For the right hand side, this follows from opening the square; for the left hand side it follows from the geometric series formula.

This discussion shows that the conjectural constants $\clq$ from Theorem~\ref{thm: L-asymptotic} are correct for $\vg{\ell} = \vg{\delta}^\chi + \vg{\delta}^\nu$ where $\chi,\nu$ are not necessarily distinct Dirichlet characters modulo $q$. One could, in principle, use Topacogullari's results from \cite{topacogullari2020fourth} to verify the analoguous constants for products of the form $L(s,\chi)L(s,\nu)$ with $\chi,\nu$ possibly having distinct moduli. 

We now return to the proof of Theorem~\ref{thm: rationalfourthmoment}. We will set $k = 2$ in (\ref{eqn: sumoverl}), and use the same classification for the different terms that arise in the right hand side of (\ref{eqn: sumoverl}) as from Section~\ref{subs: computing}. 

We state some lemmata. Their proofs are analogous to the corresponding ones from Section~\ref{subs: computing} and hence the details are omitted. 

\begin{lemma} \label{lem: major2} Suppose that $\chi$ and $\nu$ are distinct characters modulo $q$. Then, for $s = 1/2 + it$, 

\begin{equation*} \int_T^{2T} L(s,\chi)^2 \ov{L(s,\nu)}^2 \ll_q T(\log T)^2. \end{equation*}
\end{lemma}
\begin{proof} This is analogous to Lemma~\ref{lem: major}. \end{proof}

\begin{lemma} \label{lem: minor2} Suppose that $\vg{\ell}^{(1)}$ and $\vg{\ell}^{(2)}$ are tuples of nonnegative integers indexed by characters modulo $q$ satisfying $|\vg{\ell}^{(1)}| = |\vg{\ell}^{(1)}| = 2$. Further, suppose that $\vg{\ell}^{(1)}$ and $\vg{\ell}^{(2)}$ corresponds to a minor off-diagonal term. Then, for $s = 1/2 + it$, 
\begin{equation*} \int_T^{2T} \LLa(s)\ov{\LLb(s)} \,dt \ll_q T(\log T)^{3}. \end{equation*}
\end{lemma}
\begin{proof} This is analogous to Lemma~\ref{lem: minor}. \end{proof}

We can now prove the theorem. Putting $k = 2$ in (\ref{eqn: sumoverl}), we get
\begin{equation} \label{eqn: sumoverlrat} \frac{q^{2}}{\varphi(q)^{4}} \sum_{\substack{|\vg{\ell}^{(1)}| = 2,\\|\vg{\ell}^{(2)}|=2}}\binom{k}{\vg{\ell}^{(1)}} \binom{k}{\vg{\ell}^{(2)}} \left[\prod_{\chi} \chi(a)^{\ell^{(2)}_{\chi} - \ell^{(1)}_{\chi}}\right] \int_T^{2T}\LLa(s)\ov{\LLb(s)}\,dt. \end{equation}
We now use Proposition~\ref{prop: Lfourth} to deal with the terms with the primary diagonal terms (i.e. those corresponding to $\vg{\ell}^{(1)} = \vg{\ell}^{(2)} = 2\vg{\delta}^\chi$). The discussion from Section~\ref{subs: computing} tells us that summing the main terms from Proposition~\ref{prop: Lfourth} over $\chi$ contributes $c_2(\alpha) T(\log T)^{4}$, which gives the main term in Theorem~\ref{thm: rationalfourthmoment}. Note that this matches up with the conjectural constant from Theorem~\ref{thm: constant} for $k = 2$. 

It remains to show that all the remaining terms can be absorbed in the error term in Theorem~\ref{thm: rationalfourthmoment}. We do this by applying Proposition~\ref{prop: Lmean}, Lemma~\ref{lem: major2} and Lemma~\ref{lem: minor2} appropriately to terms in (\ref{eqn: sumoverlrat}), depending on their classification. There are $\ll \varphi(q)^4$ such terms in (\ref{eqn: sumoverlrat}), and they each contribute at most $\ll_q T(\log T)^3$. This completes the proof.

\bibliography{hurwitz}

\begin{thebibliography}{10}

\bibitem{anderssonthesis}
J.~Andersson.
\newblock {\em Summation formulae and zeta functions}.
\newblock PhD thesis, Stockholm University, 2006.

\bibitem{andrade2019hybrid}
J.~Andrade and A.~Shamesaldeen.
\newblock {H}ybrid {E}uler-{H}adamard {P}roduct for {D}irichlet {$L$}-functions
  with {P}rime conductors over {F}unction {F}ields.
\newblock {\em preprint, arXiv:1909.08953v1}, 2019.

\bibitem{andrade2018truncated}
J.~C. Andrade, S.~M. Gonek, and J.~P. Keating.
\newblock Truncated product representations for {$L$}-functions in the
  hyperelliptic ensemble.
\newblock {\em Mathematika}, 64(1):137--158, 2018.

\bibitem{bettin4thmoment}
S.~Bettin, H.~M. Bui, X.~Li, and M.~Radziwi\l\l.
\newblock A quadratic divisor problem and moments of the {R}iemann
  zeta-function.
\newblock {\em J. Eur. Math. Soc. (JEMS)}, 22(12):3953--3980, 2020.

\bibitem{bettin2017mean}
S.~Bettin, V.~Chandee, and M.~Radziwi\l\l.
\newblock The mean square of the product of the {R}iemann zeta-function with
  {D}irichlet polynomials.
\newblock {\em J. Reine Angew. Math.}, 729:51--79, 2017.

\bibitem{bourgain2014multiplicative}
J.~Bourgain, M.~Z. Garaev, S.~V. Konyagin, and I.~E. Shparlinski.
\newblock Multiplicative congruences with variables from short intervals.
\newblock {\em J. Anal. Math.}, 124:117--147, 2014.

\bibitem{bui2018hybridquadratic}
H.~M. Bui and A.~Florea.
\newblock Hybrid {E}uler-{H}adamard product for quadratic {D}irichlet
  {$L$}-functions in function fields.
\newblock {\em Proc. Lond. Math. Soc. (3)}, 117(1):65--99, 2018.

\bibitem{bui2015hybrid}
H.~M. Bui, S.~M. Gonek, and M.~B. Milinovich.
\newblock A hybrid {E}uler-{H}adamard product and moments of {$\zeta'(\rho)$}.
\newblock {\em Forum Math.}, 27(3):1799--1828, 2015.

\bibitem{bui2007mean}
H.~M. Bui and J.~P. Keating.
\newblock On the mean values of {D}irichlet {$L$}-functions.
\newblock {\em Proc. Lond. Math. Soc. (3)}, 95(2):273--298, 2007.

\bibitem{bui2008families}
H.~M. Bui and J.~P. Keating.
\newblock On the mean values of {$L$}-functions in orthogonal and symplectic
  families.
\newblock {\em Proc. Lond. Math. Soc. (3)}, 96(2):335--366, 2008.

\bibitem{cassels1961footnote}
J.~W.~S. Cassels.
\newblock Footnote to a note of {D}avenport and {H}eilbronn.
\newblock {\em J. London Math. Soc.}, 36:177--184, 1961.

\bibitem{conreymomentsrecipe}
J.~B. Conrey, D.~W. Farmer, J.~P. Keating, M.~O. Rubinstein, and N.~C. Snaith.
\newblock Integral moments of {$L$}-functions.
\newblock {\em Proc. London Math. Soc. (3)}, 91(1):33--104, 2005.

\bibitem{conreyghosh}
J.~B. Conrey and A.~Ghosh.
\newblock On mean values of the zeta-function.
\newblock {\em Mathematika}, 31(1):159--161, 1984.

\bibitem{conreygonek}
J.~B. Conrey and S.~M. Gonek.
\newblock High moments of the {R}iemann zeta-function.
\newblock {\em Duke Math. J.}, 107(3):577--604, 2001.

\bibitem{davenport1936zeros}
H.~Davenport and H.~Heilbronn.
\newblock On the {Z}eros of {C}ertain {D}irichlet {S}eries.
\newblock {\em J. London Math. Soc.}, 11(3):181--185, 1936.

\bibitem{djankovic2013symmetric}
G.~Djankovi\'{c}.
\newblock Euler-{H}adamard products and power moments of symmetric square
  {$L$}-functions.
\newblock {\em Int. J. Number Theory}, 9(3):621--639, 2013.

\bibitem{gonek2007hybrid}
S.~Gonek, C.~Hughes, and J.~Keating.
\newblock A hybrid {E}uler-{H}adamard product for the {R}iemann zeta function.
\newblock {\em Duke Math. J.}, 136(3):507--549, 2007.

\bibitem{gonekthesis}
S.~M. Gonek.
\newblock {\em Analytic properties of zeta and {$L$}-functions}.
\newblock PhD thesis, University of Michigan, 1979.

\bibitem{gonek1981zeros}
S.~M. Gonek.
\newblock The zeros of {H}urwitz's zeta function on {$\sigma ={1\over 2}$}.
\newblock In {\em Analytic number theory ({P}hiladelphia, {P}a., 1980)}, volume
  899 of {\em Lecture Notes in Math.}, pages 129--140. Springer, Berlin-New
  York, 1981.

\bibitem{hardy1916contributions}
G.~H. Hardy and J.~E. Littlewood.
\newblock Contributions to the theory of the riemann zeta-function and the
  theory of the distribution of primes.
\newblock {\em Acta Math.}, 41(1):119--196, 1916.

\bibitem{harper2013sharp}
A.~Harper.
\newblock {Sharp conditional bounds for moments of the Riemann zeta function}.
\newblock {\em preprint, arXiv:1305.4618}, 2013.

\bibitem{heap2014twisted}
W.~Heap.
\newblock The twisted second moment of the {D}edekind zeta function of a
  quadratic field.
\newblock {\em Int. J. Number Theory}, 10(1):235--281, 2014.

\bibitem{heap2013hybriddedekind}
W.~Heap.
\newblock Moments of the {D}edekind zeta function and other non-primitive
  {$L$}-functions.
\newblock {\em Math. Proc. Cambridge Philos. Soc.}, 170(1):191--219, 2021.

\bibitem{heap2021splitting}
W.~Heap.
\newblock On the splitting conjecture in the hybrid model for the {R}iemann
  zeta function.
\newblock {\em preprint, arXiv:2102.02092}, 2021.

\bibitem{heap2019sharp}
W.~Heap, M.~Radziwi\l\l, and K.~Soundararajan.
\newblock Sharp upper bounds for fractional moments of the {R}iemann zeta
  function.
\newblock {\em Q. J. Math.}, 70(4):1387--1396, 2019.

\bibitem{heap2022paucity}
W.~Heap, A.~Sahay, and T.~D. Wooley.
\newblock A paucity problem associated with a shifted integer analogue of the
  divisor function.
\newblock {\em Journal of Number Theory}, in press, 2022.

\bibitem{heap2022lower}
W.~Heap and K.~Soundararajan.
\newblock Lower bounds for moments of zeta and {$L$}-functions revisited.
\newblock {\em Mathematika}, 68(1):1--14, 2022.

\bibitem{heath1981fractional}
D.~R. Heath-Brown.
\newblock Fractional moments of the {R}iemann zeta function.
\newblock {\em J. London Math. Soc. (2)}, 24(1):65--78, 1981.

\bibitem{hughes2010twistedfourth}
C.~P. Hughes and M.~P. Young.
\newblock The twisted fourth moment of the {R}iemann zeta function.
\newblock {\em J. Reine Angew. Math.}, 641:203--236, 2010.

\bibitem{ingham1928mean}
A.~E. Ingham.
\newblock Mean-{V}alue {T}heorems in the {T}heory of the {R}iemann
  {Z}eta-{F}unction.
\newblock {\em Proc. London Math. Soc. (2)}, 27(4):273--300, 1927.

\bibitem{ishikawa2006difference}
H.~Ishikawa.
\newblock A difference between the values of {$|L(1/2+it,\chi_j)|$} and
  {$|L(1/2+it,\chi_k)|$}. {I}.
\newblock {\em Comment. Math. Univ. St. Pauli}, 55(1):41--66, 2006.

\bibitem{keatingsnaithzeta}
J.~P. Keating and N.~C. Snaith.
\newblock Random matrix theory and {$\zeta(1/2+it)$}.
\newblock {\em Comm. Math. Phys.}, 214(1):57--89, 2000.

\bibitem{knoppsinaifunctional}
M.~Knopp and S.~Robins.
\newblock Easy proofs of {R}iemann's functional equation for {$\zeta(s)$} and
  of {L}ipschitz summation.
\newblock {\em Proc. Amer. Math. Soc.}, 129(7):1915--1922, 2001.

\bibitem{mertensforapconstant}
A.~Languasco and A.~Zaccagnini.
\newblock A note on {M}ertens' formula for arithmetic progressions.
\newblock {\em J. Number Theory}, 127(1):37--46, 2007.

\bibitem{milinovich2014moments}
M.~B. Milinovich and C.~L. Turnage-Butterbaugh.
\newblock Moments of products of automorphic {$L$}-functions.
\newblock {\em J. Number Theory}, 139:175--204, 2014.

\bibitem{radziwill2013continuous}
M.~Radziwi{\l\l} and K.~Soundararajan.
\newblock Continuous lower bounds for moments of zeta and {$L$}-functions.
\newblock {\em Mathematika}, 59(1):119--128, 2013.

\bibitem{ramachandra1978some}
K.~Ramachandra.
\newblock Some remarks on the mean value of the {R}iemann zeta function and
  other {D}irichlet series. {I}.
\newblock {\em Hardy-Ramanujan J.}, 1:15, 1978.

\bibitem{ramachandra1980some}
K.~Ramachandra.
\newblock Some remarks on the mean value of the {R}iemann zeta function and
  other {D}irichlet series. {II}.
\newblock {\em Hardy-Ramanujan J.}, 3:1--24, 1980.

\bibitem{ramachandra1980some2}
K.~Ramachandra.
\newblock Some remarks on the mean value of the {R}iemann zeta function and
  other {D}irichlet series. {III}.
\newblock {\em Ann. Acad. Sci. Fenn. Ser. A I Math.}, 5(1):145--158, 1980.

\bibitem{rane1980hurwitz}
V.~V. Rane.
\newblock On the mean square value of {D}irichlet {$L$}-series.
\newblock {\em J. London Math. Soc. (2)}, 21(2):203--215, 1980.

\bibitem{soundararajan2009moments}
K.~Soundararajan.
\newblock Moments of the {R}iemann zeta function.
\newblock {\em Ann. of Math. (2)}, 170(2):981--993, 2009.

\bibitem{spira1976zeros}
R.~Spira.
\newblock Zeros of {H}urwitz zeta functions.
\newblock {\em Math. Comp.}, 30(136):863--866, 1976.

\bibitem{titchmarsh}
E.~Titchmarsh.
\newblock {\em The theory of the {R}iemann zeta-function}.
\newblock The Clarendon Press, Oxford University Press, New York, second
  edition, 1986.
\newblock Edited and with a preface by D. R. Heath-Brown.

\bibitem{topacogullari2020fourth}
B.~Topacogullari.
\newblock The fourth moment of individual {D}irichlet {$L$}-functions on the
  critical line.
\newblock {\em Math. Z.}, 298(1-2):577--624, 2021.

\bibitem{voronin1976zeros}
S.~M. Voronin.
\newblock The zeros of zeta-functions of quadratic forms.
\newblock {\em Trudy Mat. Inst. Steklov.}, 142:135--147, 269, 1976.
\newblock Number theory, mathematical analysis and their applications.

\bibitem{mertensforapwilliams}
K.~S. Williams.
\newblock Mertens' theorem for arithmetic progressions.
\newblock {\em J. Number Theory}, 6:353--359, 1974.

\bibitem{wu2019twisted}
X.~Wu.
\newblock The twisted mean square and critical zeros of {D}irichlet
  {$L$}-functions.
\newblock {\em Math. Z.}, 293(1-2):825--865, 2019.

\bibitem{zhan1992mean}
T.~Zhan.
\newblock On the mean square of {D}irichlet {$L$}-functions.
\newblock {\em Acta Math. Sinica (N.S.)}, 8(2):204--224, 1992.
\newblock A Chinese summary appears in Acta Math. Sinica {{\bf{3}}6} (1993),
  no. 3, 432.

\end{thebibliography}
\bibliographystyle{abbrv}

\end{document}